\documentclass[10pt]{article}
\usepackage[a4paper]{geometry}
\usepackage{amsmath,amsfonts,amssymb,amsthm}
\usepackage[english]{babel}

\usepackage{tikz}
\usepackage[all]{xy}
\usepackage{caption}
\usepackage{graphicx}
\usepackage{mathrsfs}

\usepackage{pifont}
\usepackage{shuffle}

\usepackage[colorlinks=true]{hyperref}

\usepackage{color, colortbl}
\definecolor{paleblue}{rgb}{0.7, 0.7, 1.0}
\definecolor{palegreen}{rgb}{0.7, 1,0.7}

  \theoremstyle{plain}
\newtheorem{theorem}{Theorem}
\newtheorem{proposition}{Proposition}[section]
\newtheorem{corollary}[proposition]{Corollary}

  \theoremstyle{remark}
\newtheorem{remark}[proposition]{Remark}
  \theoremstyle{definition}
\newtheorem{definition}[proposition]{Definition}

\newtheorem{lemma}[proposition]{Lemma}
\newtheorem{example}[proposition]{Example}

\newcommand{\sh}{\underset{\sigma}{\shuffle}}
\newcommand{\Sh}{\underset{-\sigma}{\shuffle}}

\newcommand{\csh}{\underset{\sigma}{\cshuffle}}
\newcommand{\Csh}{\underset{-\sigma}{\cshuffle}}

\newcommand{\ShAlg}{\operatorname{Sh}}
\newcommand{\ShCoalg}{\operatorname{\overline{Sh}}}
\newcommand{\ssigma}{\boldsymbol{\sigma}}
\newcommand{\wlhd}{\widetilde{\lhd}}
\newcommand{\II}{\mathbf{I}}

\newcommand{\ZZ}{\mathbb{Z}}

\newcommand{\C}{\mathcal C}

\newcommand{\Hom}{\operatorname{Hom}}
\newcommand{\End}{\operatorname{End}}

\newcommand{\Ob}{\operatorname{Ob}}
\newcommand{\Set}{\mathbf{Set}}
\renewcommand{\k}{\Bbbk}
\newcommand{\Vect}{\mathbf{Vect}_\k}

\newcommand{\VectGrad}{{_\ZZ}\!\mathbf{Vect}_\k}
\newcommand{\kk}{R}
\newcommand{\kVect}{\mathbf{Mod}_\kk}

\newcommand{\Alg}{\mathbf{UAlg}}
\newcommand{\Lei}{\mathbf{ULei}}
\newcommand{\coAlg}{\mathbf{coUAlg}}
\newcommand{\coLei}{\mathbf{coULei}}
\newcommand{\Br}{\mathbf{Br}}

\newcommand{\Id}{\operatorname{Id}}
\newcommand{\Ker}{\operatorname{Ker}}

\newcommand{\op}{\operatorname{op}}
\newcommand{\opp}{\otimes^{\op}}
\newcommand{\oppp}{\op,\opp}
\newcommand{\one}{\mathbf{1}}
\newcommand{\wV}{\widetilde{V}}

\newcommand{\ov}{\overline{v}}

\renewcommand{\le}{\leqslant}
\renewcommand{\ge}{\geqslant}

\newcommand{\ii}{\ding{250}}

\begin{document}

\title{Homologies of Algebraic Structures via Braidings and Quantum Shuffles}

\author{Victoria Lebed \\ \itshape lebed@math.jussieu.fr}

\maketitle

\begin{abstract}
\footnotesize
In this paper we construct ``structural'' pre-braidings characterizing different algebraic structures: a rack, an associative algebra, a Leibniz algebra and their representations. Some of these pre-braidings seem original. On the other hand, we propose a general homology theory for pre-braided vector spaces and braided modules, based on the quantum co-shuffle comultiplication. Applied to the structural pre-braidings above, it gives a generalization and a unification of many known homology theories. All the constructions are categorified, resulting in particular in their super- and co-versions. Loday's hyper-boundaries, as well as certain homology operations are efficiently treated using the ``shuffle'' tools. 
\end{abstract}



\tableofcontents

\section{Introduction}

The aim of this paper is to develop a unifying framework for (co)homologies of algebraic structures. Our starting point is the following fundamental procedure:

\begin{center}
\fbox{algebraic structure $\leadsto$ chain complex.}
   \captionof{figure}{Homology of algebraic structures}\label{pic:HomAlgStr}
\end{center}

The step $\leadsto$ is far from being canonical, and can be dictated by motivations of very different nature: structure deformations and obstructions, classification questions, topological applications etc. Here we propose to regard it from a purely combinatorial viewpoint, in the spirit of operad theory. The complexes one associates in practice to basic algebraic structures on a vector space $V$ usually have the same flavor: they are all \textbf{signed} sums $d_n:V^{\otimes n} \rightarrow V^{\otimes (n-1)}$ of \textbf{terms of the same nature} $d_{n;i},$ one for each component $1,2,\ldots,n$ of $V^{\otimes n}.$

The examples we have in mind are
\begin{enumerate}
\item \textit{Koszul complex} for vector spaces,
\item \textit{bar and Hochschild complexes} for associative algebras,
\item \textit{Chevalley-Eilenberg complex} for Lie algebras,
\item \textit{rack and  shelf (or one-term distributive) complexes} for self-distributive (= SD) structures.
\end{enumerate}

Verifying that one has indeed a differential, i.e. $d_{n-1} \circ d_n=0,$ can be reduced to checking some \textbf{local} algebraic identities (which mysteriously coincide with the defining properties for our algebraic structure!) coupled with a sign manipulation, no less mysterious. 

For many algebraic structures, their chain complexes can be refined by introducing a (weakly) (pre)(bi)\textbf{simplicial structure} on $T(V)$ (see section \ref{sec:simplicial} or J.-L.Loday's book \cite{Cyclic} for the simplicial vocabulary). Moreover, the degree $-1$ differentials can be generalized to \textbf{Loday's hyper-boundaries} of arbitrary degree (see the definitions from section \ref{sec:hyper}, or exercise E.2.2.7  in \cite{Cyclic}, from which this notion takes inspiration). Some \textbf{homology operations}, similar for different algebraic structures, are also to be mentioned here. Such common features are presented, for the example of associative and SD structures, in J.Przytycki's paper \cite{Prz1}.

In this work, we propose to interpret and partially explain these parallels (typed in bold letters above) and mysteries by adding a new step to the scheme in figure \ref{pic:HomAlgStr}:

\begin{center}
\fbox{{\large algebraic structure} 
\fcolorbox{brown}{palegreen}{ ${\overset{\text{ case by case}}{\leadsto}}$ \textbf{pre-braiding}
${\overset{\text{ theorem \ref{thm:cuts}}}{\leadsto}}$ }
  {\large chain complex.}}
   \captionof{figure}{Homology of algebraic structures via pre-braidings}\label{pic:HomAlgStrBr}
\end{center}

After a short reminder on braided structures in section \ref{sec:braid}, we proceed to describing in detail the right part of this new scheme. More precisely, given a vector space endowed with a \textit{{pre-braiding}} $\sigma: V  \otimes V \rightarrow V \otimes V$ satisfying the \textit{{Yang-Baxter equation}} (=YBE)
$$ (\sigma\otimes\Id_V)\circ (\Id_V\otimes\sigma)\circ (\sigma\otimes\Id_V)=(\Id_V\otimes\sigma)\circ (\sigma\otimes\Id_V)\circ (\Id_V\otimes\sigma) \in \End(V \otimes V  \otimes V),$$
we associate in theorem \ref{thm:cuts} a bidifferential on $T(V)$ to any couple of \textit{{braided characters}} (= elements of $V^*$ ``respecting'' $\sigma$) $\epsilon$ and $\zeta,$ using {\textit{quantum co-shuffle comultiplication}} techniques (cf. M.Rosso's pioneer papers \cite{Rosso1Short},\cite{Rosso2}). We call such {(bi)differentials} \textit{braided}. 

In theorem \ref{thm:BraidedSimplHom} we refine these braided bidifferential structures: we show that they come from a pre-bisimplicial structure on $T(V),$ completed to a weakly bisimplicial one if $V$ is moreover endowed with a coassociative $\sigma$-cocommutative \textit{{comultiplication}} $\Delta$ compatible with the pre-braiding $\sigma.$ This is done using the graphical calculus (in the spirit of J.C.Baez (\cite{Baez}), S.Majid (\cite{Majid}) and other authors), appearing naturally due to our use of ``braided'' techniques. Here are for example the components of the weakly bisimplicial structure from the theorem (all diagrams are to be read from bottom to top here):

\begin{center}
\begin{tikzpicture}[scale=0.3]
 \draw (0,1) --(0,-1.6);
 \draw (0,-2.4) --  (0,-7);
 \node at (-4,-3) {$d_{n;i}=$};
 \draw[thick,red,rounded corners] (-1,0) -- (-1,-1) -- (4,-6) -- (4,-7);
 \fill[violet] (-1,0) circle (0.1);
 \node at (-1,0)  [above, violet] {$\epsilon$};
 \draw (3,1) -- (3,-4.6);
 \draw (3,-5.4) -- (3,-7);
 \draw (5,1) -- (5,-7);
 \draw (8,1) -- (8,-7);
 \node at (0.3,-2.3) [ above right] {$\sigma$};
 \node at (3.3,-5.3) [above right] {$\sigma$};
 \node at (0,-7) [below] {$\scriptstyle 1$};
 \node at (1.5,-7) [below] {$\scriptstyle\ldots$};
 \node at (4,-7) [below] {$\scriptstyle i$};
 \node at (6.5,-7) [below] {$\scriptstyle \ldots$};
 \node at (8,-7) [below] {$\scriptstyle{n}$}; 
\end{tikzpicture}
\begin{tikzpicture}[scale=0.3]
 \draw (0,1) -- (0,-7);
 \node at (-9,-3) {};
 \node at (-4,-3) {$d'_{n;i}=$};
 \draw[thick,red,rounded corners] (9,0) -- (9,-1) -- (8.3,-1.7);
 \draw[thick,red,rounded corners] (7.7,-2.3) -- (5.3,-4.7);
 \draw[thick,red,rounded corners] (4.7,-5.3)--(4,-6) -- (4,-7);
 \fill[violet] (9,0) circle (0.1);
 \node at (9,0) [above,violet] {$\zeta$};
 \draw (3,1) -- (3,-7);
 \draw (5,1) -- (5,-7);
 \draw (8,1) -- (8,-7);
 \node at (5,-5) [right] {$\sigma$};
 \node at (8,-2) [right] {$\sigma$};
 \node at (0,-7) [below] {$\scriptstyle 1$};
 \node at (1.5,-7) [below] {$\scriptstyle \ldots$};
 \node at (4,-7) [below] {$\scriptstyle i$};
 \node at (6.5,-7) [below] {$\scriptstyle\ldots$};
 \node at (8,-7) [below] {$\scriptstyle{n}$}; 
\end{tikzpicture}
\end{center}
\begin{center}
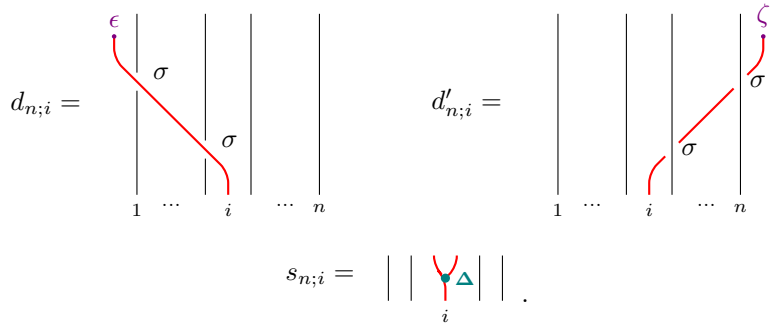

\begin{tikzpicture}[scale=0.3]
 \node at (-7,1) {};
 \node at (-4,1) {$s_{n;i}=$};
 \draw (-1,0) -- (-1,2);
 \draw (0,0) -- (0,2);
 \draw (4,0) -- (4,2);
 \draw (3,0) -- (3,2);
 \draw[thick,red,rounded corners] (1.5,0) -- (1.5,1) -- (1,1.5) -- (1,2);
 \draw[thick,red,rounded corners] (1.5,1) -- (2,1.5) -- (2,2);
 \fill[teal] (1.5,1) circle (0.2);
 \node at (1.5,1) [right,teal] {$\scriptstyle \mathbf{\Delta}$};
 \node at (1.5,0) [below] {$\scriptstyle i$};
 \node at (5,0) {.};
\end{tikzpicture}
   \captionof{figure}{Simplicial structure for braided homology}\label{pic:SimplIntro}
\end{center}
See table \ref{tab:2Approches} for a comparison of the quantum co-shuffle and the graphical approaches to braided differentials. Braided differentials are generalized to hyper-boundaries in section \ref{sec:hyper}, and some homology operations for them are studied in section \ref{sec:HomOper}.

Note that we never demand $\sigma$ to be {invertible}, emphasizing it in the term \underline{pre-}braiding. Rare in literature, this elementary generalization of the notion of braiding allows interesting examples.

\medskip
Armed with this general homology theory for pre-braided vector spaces, we are now interested in the left part of figure \ref{pic:HomAlgStrBr}. Unfortunately we do not know any systematic way of associating a pre-braiding to an algebraic structure. So we do it by hand in section \ref{sec:examples} for each of the four structures in the list above. In each case, the pre-braiding $\sigma$ we propose \underline{encodes} surprisingly well the structure in question, in the sense that
\begin{itemize}
\item YBE for $\sigma$ is equivalent to the defining relation for the structure (e.g. the associativity for an algebra), under some mild assumptions concerning units;
\item the invertibility condition for $\sigma,$ when this makes sense, translates important algebraic properties (e.g. the rack condition);
\item braided morphisms (= those preserving $\sigma$) are precisely algebra morphisms (= those preserving the underlying structure);
\item braided characters  for $\sigma$ include the usual characters for the structure (e.g. algebra characters);
\item the comultiplication necessary for constructing the degeneracies turns out to be quite characteristic of the structures.
\end{itemize}
Thus the left part of the scheme in figure \ref{pic:HomAlgStrBr} can be informally stated in a stronger way:
\begin{center}
\fbox{``algebraic structure $=$ pre-braiding''.}
\end{center}

\medskip
Besides its unifying character, our braided homology theory has the following advantages:
\begin{enumerate}
\item We produce two compatible differentials -- a left and a right one -- for each braided character, these differentials often being compatible even for different characters. Their combinations thus give a {family of homology theories} for the same algebraic structure. For instance, for SD structures we recover
\begin{enumerate}
\item usual shelf (\cite{PrzSikora}, \cite{Prz1}), rack (\cite{RackHom}) and  quandle (\cite{QuandleHom}) homology theories;
\item the partial derivatives from \cite{Prz2};
\item and the twisted rack homology from \cite{TwistedQuandle}.
\end{enumerate}
\item The technical sign manipulation (especially heavy for the Chevalley-Eilenberg complex) is controlled either by using the negative pre-braiding $-\sigma$ in the quantum co-shuffle comultiplication, or by counting the number of intersections in the graphical interpretation.
\item The identities $d_{n-1} \circ d_n=0,$ which are of ``global'' nature, are replaced with the YBE for the corresponding pre-braiding, which is ``local'' and thus easier to verify.
\item The decomposition $d_n=\sum_{i=1}^n (-1)^{i-1} d_{n;i}$ becomes natural when one reasons in terms of braids and strands.
\item So do some homology operations -- for instance, the generalizations of the homology operations for shelves, defined by M.Niebrzydowski and J.Przytycki in \cite{Prz2}.
\item Subscript chasing (in the relations defining simplicial structures for instance) is substituted with the more transparent ``strand chasing''.
\item J.-L.Loday's hyper-boundaries of degree $-i$ arise naturally in the co-shuffle interpretation: one simply replaces the $V^{\otimes n} \rightarrow V^{\otimes 1} \otimes V^{\otimes (n-1)}$ component of the quantum co-shuffle coproduct with the $V^{\otimes n} \rightarrow V^{\otimes i} \otimes V^{\otimes (n-i)}$ component.
\end{enumerate}

We thus recover all the common features of different algebraic homology theories mentioned (in bold letters) in the beginning of this introduction. Moreover, we obtain a simplified and conceptual way of proving that $d^2=0,$ as well as of ``guessing'' the right boundary map. As an illustration, note that the ``braided'' considerations naturally lead one to lifting the Chevalley-Eilenberg complex from the external to the tensor algebra of a Lie algebra, reinterpreting the work of J.-L.Loday on \textit{{Leibniz algebras}} (\cite{Cyclic}) from a more conceptual viewpoint. We also recover for Leibniz algebras the braiding studied in the Lie case by A.Crans (\cite{Crans}, \cite{CatSelfDistr}). 

\medskip

In section \ref{sec:coeffs} we introduce the notion of a \textit{{braided module}} over a pre-braided space $(V, \sigma),$ unifying, quite unexpectedly, the familiar notions of modules over different structures. Concretely, it is a space $M$ equipped with a linear map $\rho:M\otimes V \rightarrow M$ satisfying
$$ \rho \circ (\rho \otimes \Id_V)=\rho \circ (\rho \otimes \Id_V)\circ (\Id_M\otimes \sigma):M\otimes V\otimes V\rightarrow M.$$
These braided modules are natural candidates for coefficients in the braided complexes, leading us to \textit{{braided homologies with coefficients}}, recovering for example the Hochschild and Chevalley-Eilenberg homologies.

\medskip
 
 We now give a table of braided structures encoding the algebraic structures mentioned above, and of the familiar complexes recovered as particular cases of our braided complexes.
 
 \begin{center}
\begin{tabular}{|c|c|c|c|c|}
\hline 
\textbf{structure} & \textbf{pre-braiding} &\textbf{invertibility} & \textbf{$\Delta$} & \textbf{complexes} \\
\hline 
vector space $V$ & flip $\tau: v \otimes w \mapsto w \otimes v$ & $\tau^{-1} = \tau$ & -- & Koszul \\
\hline 
unital associative  & $\sigma_{\mu}:v\otimes w  \mapsto $ & no inverse& $\Delta(v)= \one \otimes v$  & bar,\\
algebra $(V,\cdot,\one)$ & $ \one \otimes vw$ &in general & &Hochschild\\ 
\hline 
unital Leibniz & $\sigma_{[,]}: v\otimes w  \mapsto $ & & $\Delta(v)= $ & Leibniz,\\
 algebra $(V,[,],\one)$ & $w \otimes v$ &$\exists\;\sigma_{[,]}^{-1}$ & $v\otimes \one + \one \otimes v,$ & Chevalley-\\
 & $ + \one \otimes [v, w]$ && $\Delta(\one)= \one\otimes \one$ & Eilenberg\\
 \hline 
shelf $(S,\lhd),$ & $\sigma_{\lhd}$ : &$\exists\;\sigma_{\lhd}^{-1}$ iff & $\Delta(a)= (a,a)$ & shelf, quandle,\\
$V:=\k S$ & $(a,b) \mapsto (b,a \lhd b)$   & $S$ is a rack & & (twisted) rack \\
\hline
\end{tabular} 
\end{center}

\begin{center}
\begin{tabular}{|c|c|c|c|}
\hline 
\textbf{structure} & \textbf{braided characters}  & \textbf{braided modules} & \textbf{homology operations}\\
\hline 
vector & any $\epsilon \in V^*$ & space endowed with &   multiplication \\
space &  &commutative operators & by scalars \\
\hline 
unital &algebra character:& algebra module:& peripheral:  \\
associative& $\epsilon(vw)=\epsilon(v)\epsilon(w),$  & $m \cdot vw  =$ & ${^\epsilon}\!\pi_w(v_1 \ldots v_{n-1} v_n) =$  \\
algebra & $ \epsilon(\one)=1$ & $(m \cdot v) \cdot w$ & $v_1 \ldots v_{n-1} (v_n w)$\\  
\hline 
unital & Lie character: & Leibniz module (\cite{Cyclic}):  &  adjoint:\\
Leibniz & $\epsilon([v,w])=0,$ & $m \cdot [v,w]  =$  & ${^\epsilon}\!\pi_w(v_1 \ldots v_n) =$  \\
 algebra & $ \epsilon(\one)=1$ & $(m \cdot v) \cdot w - (m \cdot w) \cdot v$ & $\sum_{i=1}^n v_1 \ldots [v_i,w] \ldots v_n$\\
 \hline 
 & shelf character: & shelf module (\cite{ChangNelson}, \cite{Kamada3}): & diagonal:\\
shelf & $ \epsilon (a\lhd b) = \epsilon (a)$&  $(m \cdot a) \cdot b =$  &  ${^\epsilon}\!\pi_b(a_1, \ldots, a_n) =$ \\
  & &$(m \cdot b) \cdot (a \lhd b)$ & $(a_1\lhd b, \ldots, a_n\lhd b)$\\
 \hline
\end{tabular} 
   \captionof{table}{Braided homology ingredients in concrete algebraic settings}\label{tab:BasicIngredients}
\end{center}

Pre-braidings for vector spaces and self-distributive structures are classical; that for Leibniz algebras was used in the Lie case by A.Crans in \cite{Crans} (cf. also \cite{CatSelfDistr}), but does not seem to be widely known; the author has never met the pre-braiding for associative algebras elsewhere.

\medskip

Our braided homology theory, as well as the pre-braidings for associative and Leibniz algebras, are raised to the \textit{{categorical level}} in section \ref{sec:cat}. The \textit{categorification of SD structures}  and of the corresponding pre-braiding is less straightforward and is done in \cite{Lebed3} (see also \cite{CatSelfDistr} for an alternative construction). Several typical applications of the categorical approach are presented, obtained by changing the underlying category (e.g. the {Leibniz \underline{super}algebra} homology) or using different types of categorical dualities (e.g. cobar and Cartier complexes for \underline{co}algebras).

An important feature of our categorified pre-braiding, besides relaxing the invertibility condition, is its \textit{{``local''}} character: instead of demanding the whole category to be pre-braided, we need a pre-braiding for a single object only, omitting in particular the naturality condition. 

\medskip
This paper contains the results from the first part of the author's thesis \cite{Lebed}, where more details and proofs can be found.

\medskip
\underline{\textbf{Notations and conventions.}}

We systematically use notation $\kk$ for a commutative unital ring, and $\k$ for a field. The word ``linear'' means $\kk$- (or $\k$-) linear, and all tensor products are over $\kk$ (or $\k$), unless we work in the settings of a general monoidal category. The category of $\kk$-modules (resp. $\k$-vector spaces) and linear maps is denoted by $\kVect$ (resp. $\Vect$).

 Notation $T(V):=\bigoplus_{n \geq 0}V^{\otimes n}$ is used for the tensor algebra of an $\kk$-module $V,$ with $V^{\otimes 0}:=\kk.$ A simplified notation is used for its elements:
$$ \overline{v} = v_1 v_2 \ldots  v_n := v_1\otimes v_2\otimes\ldots\otimes v_n \in V^{\otimes n},$$
leaving the tensor product sign for 
$$ v_1 v_2 \ldots  v_n \otimes w_1 w_2 \ldots w_m \in V^{\otimes n}\otimes W^{\otimes m}.$$

We often call the $\kk$-module $T(V)$ the \textbf{\emph{tensor module}} of $V$, emphasizing that it can be endowed with a multiplication different from the usual concatenation. We talk about the \textbf{\emph{tensor space}} of $V$ in the $\k$-linear setting. 

The notation $V^{\otimes n}$ is sometimes reduced to $V^{n},$ and $\Id_{V^{\otimes n}}$ to $\Id_n.$

Given an $\kk$-module $V$ and a linear map $\varphi:V^{\otimes l}\rightarrow V^{\otimes r},$ the following notations are repeatedly used:
\begin{equation}\label{eqn:phi_i}
\varphi_i := \Id_{i-1}\otimes\varphi \otimes \Id_{k-i+1}:V^{k+l}\rightarrow V^{k+r},
\end{equation}
\begin{equation}\label{eqn:phi^i}
\varphi^n := (\varphi_1)^{\circ n}=\varphi_1 \circ \cdots \circ\varphi_1 :V^{ k}\rightarrow V^{k+n(r-l)},
\end{equation}
where $\varphi_1$ is composed with itself $n$ times. Similar notations are used for tensor products of different modules.

The above notations are also used (when they make sense) in the context of a strict monoidal category.

By a \emph{differential} on a graded $\kk$-module (for example $T(V)$) we mean a square zero endomorphism of degree $+1$ or $-1,$ while a \emph{bidifferential} is a pair of anticommuting differentials. The word \emph{complex} always means a differential (co)chain complex here, i.e. a graded $\kk$-module endowed with a differential. Similarly, a \emph{bicomplex} is a graded $\kk$-module endowed with a bidifferential.

\section{Braided world: a short reminder}\label{sec:braid}

We recall here various facts about \textit{braided vector spaces} necessary for subsequent sections. For a more systematic treatment of braid groups, \cite{Braids} is an excellent reference. A particular focus is made here on \textit{quantum (co-)shuffles}, introduced and studied by M.Rosso in \cite{Rosso2} and \cite{Rosso1}. These structures will provide an important tool for constructing braided space homologies in the next section. 

 All the notions defined here for vector spaces are directly generalized for $\kk$-modules. We prefer the language of vector spaces for its familiarity.

\begin{definition}
\begin{itemize}
\item 
A \emph{pre-braiding} on a $\k$-vector space $V$ is a linear map  $\sigma:V \otimes V \rightarrow V \otimes V$ satisfying the \emph{Yang-Baxter equation} (abbreviated as YBE)
\begin{equation}
  \tag{YB}
  \sigma_{1}\circ \sigma_{2}\circ \sigma_{1}=\sigma_{2}\circ \sigma_{1}\circ \sigma_{2} : V \otimes V  \otimes V \longrightarrow V \otimes V \otimes V,
  \label{eqn:YB}
\end{equation}
where $\sigma_{i}$ is the braiding $\sigma$ applied to components $i$ and $i+1$ of $V^{\otimes 3}$ (cf. notation \eqref{eqn:phi_i}). 

\item A \emph{braiding} is an invertible pre-braiding.

\item A braiding is called \emph{symmetric} if $\sigma^2 = \Id_{V \otimes V }.$

\item A vector space endowed with a (pre-)braiding is called \emph{(pre-)braided}.

\item A \emph{braided morphism} between pre-braided spaces $(V, \sigma_V)$ and $(W, \sigma_W)$ is a $\k$-linear map $f:V \rightarrow W$ respecting the pre-braidings:
$$ (f \otimes f) \circ \sigma_V = \sigma_W \circ (f \otimes f) : V \otimes V \rightarrow W \otimes W.$$
\end{itemize}
\end{definition}

Note that unlike most authors \textbf{we work with pre-braidings}, allowing interesting highly non-invertible examples.

\begin{remark}\label{rmk:SetYB}
A (pre-)braiding on a set is defined similarly: tensor products $\otimes$ are simply replaced by Cartesian products $\times.$ These two settings are particular cases of a more abstract one: they both come from \textbf{\textit{(pre-)braided categories}}, studied in detail in section \ref{sec:cat}.
\end{remark}

\begin{example}
The most familiar braidings are the \emph{\textbf{flip}},  the \emph{\textbf{signed flip}} and their generalization for graded vector spaces, the \emph{\textbf{Koszul flip}}:
\begin{align}
\tau : v\otimes w &\longmapsto w\otimes v,\notag \\
-\tau : v\otimes w &\longmapsto - w\otimes v,\notag \\
\tau_{Koszul} : v\otimes w &\longmapsto (-1)^{\deg v \deg w} w\otimes v \label{eqn:KoszulFlip}
\end{align}
for homogeneous $v$ and $w.$ The last braiding explains the {\emph{Koszul sign convention}} in many settings.
\end{example}

\begin{remark} In general for a (pre-)braiding $\sigma,$ its opposite $-\sigma:v\otimes w\mapsto -\sigma (v\otimes w)$ is also a (pre-)braiding.
\end{remark} 

\medskip
A pre-braiding gives an action of the positive braid monoid $B^+_n$ on $V^{\otimes n},$ i.e. a monoid morphism
\begin{align}
\rho: B^+_n & \longrightarrow \End_{\k}(V^{\otimes n}),\notag \\
b &\longmapsto b^{\sigma} \label{eqn:^sigma}
\end{align}
defined on the generators $\sigma_i$ of $B^+_n$ by 
\begin{equation}\label{eqn:BnAction}
\sigma_i \longmapsto \Id_{i-1} \otimes \sigma \otimes \Id_{n-i-1}.
\end{equation}
 This action is best depicted in the graphical form
\begin{center}
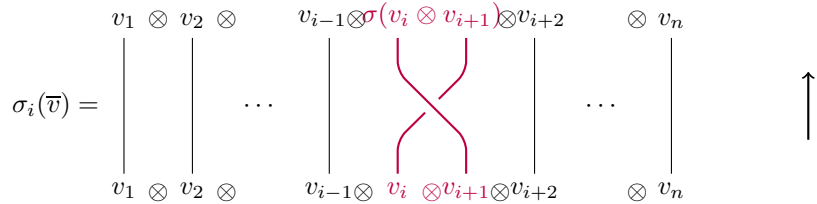

\begin{tikzpicture}[scale=0.9]
\node at (-1,1) {$\sigma_i(\overline{v}) =$};
\draw (0,0)--(0,2);
\draw (1,0)--(1,2);
\node at (2,1) {$\cdots$};
\draw (3,0)--(3,2);
\draw [rounded corners, thick, purple] (5,0) -- (5,0.5) -- (4,1.5) --(4,2);
\draw [rounded corners, thick, purple] (4,0) -- (4,0.5) -- (4.4,0.9);
\draw [rounded corners, thick, purple] (5,2) -- (5,1.5) -- (4.6,1.1);
\draw (6,0)--(6,2);
\node at (7,1) {$\cdots$};
\draw (8,0)--(8,2);
\draw [->, thick]  (10,0.5) -- (10,1.5);
\node [below] at (0,0) {$v_1$};
\node [below] at (1,0) {$v_2$};
\node [below] at (3,0) {$v_{i-1}$};
\node [below,purple] at (4,0) {$v_i$};
\node [below,purple] at (5,0) {$v_{i+1}$};
\node [below] at (6,0) {$v_{i+2}$};
\node [below] at (8,0) {$v_n$};
\node [above] at (0,2) {$v_1$};
\node [above] at (1,2) {$v_2$};
\node [above] at (2.9,2) {$v_{i-1}$};
\node [above] at (6.1,2) {$v_{i+2}$};
\node [above] at (8,2) {$v_n$};
\node [above,purple] at (4.5,2) {${\sigma(v_i\otimes v_{i+1})}$}; 
\node [below] at (0.5,0) {$\otimes$};
\node [below] at (1.5,0) {$\otimes$};
\node [below] at (3.5,0) {$\otimes$};
\node [below,purple] at (4.5,0) {$\otimes$};
\node [below] at (5.5,0) {$\otimes$};
\node [below] at (7.5,0) {$\otimes$};
\node [above] at (0.5,2) {$\otimes$};
\node [above] at (1.5,2) {$\otimes$};
\node [above] at (3.4,2) {$\otimes$};
\node [above] at (5.6,2) {$\otimes$};
\node [above] at (7.5,2) {$\otimes$};
\end{tikzpicture}
\captionof{figure}{$B^+_n$ acts via pre-braidings}\label{pic:BraidedAction}
\end{center} 

All diagrams in this work are to be read \textbf{from bottom to top}, as indicated by the arrow on the diagram above. One could have presented the crossing as \begin{tikzpicture}[scale=0.5]
 \draw (0,0) -- (1,1);
 \draw (0,1) -- (0.3,0.7);
 \draw (1,0) -- (0.7,0.3);
 \node  at (1.2,0) {};
\end{tikzpicture}, which is often done in literature. It is just a matter of convention, and the one used here comes from rack theory (section \ref{sec:racks}).

For braidings, the action above is in fact an action of the braid group $B_n,$ and for symmetric braidings it is an action of the symmetric group $S_n.$

The graphical translation of the YBE for pre-braidings is the third Reidemeister move, which is at the heart of knot theory:
\begin{center}
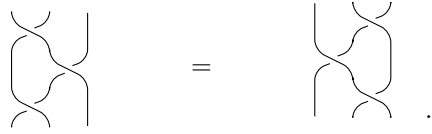

\begin{tikzpicture}[scale=0.5]
\draw [rounded corners](0,0)--(0,0.25)--(0.4,0.4);
\draw [rounded corners](0.6,0.6)--(1,0.75)--(1,1.25)--(1.4,1.4);
\draw [rounded corners](1.6,1.6)--(2,1.75)--(2,3);
\draw [rounded corners](1,0)--(1,0.25)--(0,0.75)--(0,2.25)--(0.4,2.4);
\draw [rounded corners](0.6,2.6)--(1,2.75)--(1,3);
\draw [rounded corners](2,0)--(2,1.25)--(1,1.75)--(1,2.25)--(0,2.75)--(0,3);
\node  at (5,1.5){$=$};
\end{tikzpicture}
\begin{tikzpicture}[scale=0.5]
\node  at (-2,1.5){};
\draw [rounded corners](1,1)--(1,1.25)--(1.4,1.4);
\draw [rounded corners](1.6,1.6)--(2,1.75)--(2,3.25)--(1,3.75)--(1,4);
\draw [rounded corners](0,1)--(0,2.25)--(0.4,2.4);
\draw [rounded corners](0.6,2.6)--(1,2.75)--(1,3.25)--(1.4,3.4);
\draw [rounded corners](1.6,3.6)--(2,3.75)--(2,4);
\draw [rounded corners](2,1)--(2,1.25)--(1,1.75)--(1,2.25)--(0,2.75)--(0,4);
\node  at (3,1){.};
\end{tikzpicture}
\captionof{figure}{Yang-Baxter equation $=$ Reidemeister move III}\label{pic:YB}
\end{center}

\medskip
Numerous constructions become natural in the graphical settings. For instance,

\begin{remark}\label{rmk:br_tensor}
A (pre-)braiding $\sigma$ on $V$ naturally extends to a (pre-)braiding $\ssigma$ on its tensor space $T(V)$ via
$$\ssigma(\overline{v}\otimes\overline{w}) = (\sigma_k\cdots\sigma_1)\cdots(\sigma_{n+k-2}\cdots\sigma_{n-1})(\sigma_{n+k-1}\cdots\sigma_{n}) (\overline{v}\overline{w}) \in V^{\otimes k}\otimes V^{\otimes n}$$
for  pure tensors $\overline{v} \in V^{\otimes n}, \overline{w} \in V^{\otimes k}$ ($\overline{v}\overline{w}$ being simply their concatenation),
or, graphically,
\begin{center}
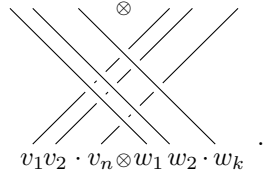

\begin{tikzpicture}[scale=0.3]
\node  at (5,6){$\scriptstyle{\otimes}$};
\draw (6,0)--(0,6);
\draw (7,0)--(1,6);
\draw (9,0)--(3,6);
\draw (1,0)--(3.3,2.3);
\draw (3.7,2.7)--(3.8,2.8);
\draw (4.2,3.2)--(4.8,3.8);
\draw (5.2,4.2)--(7,6);
\draw (2,0)--(3.8,1.8);
\draw (4.2,2.2)--(4.3,2.3);
\draw (4.7,2.7)--(5.3,3.3);
\draw (5.7,3.7)--(8,6);
\draw (4,0)--(4.8,0.8);
\draw (5.2,1.2)--(5.3,1.3);
\draw (5.7,1.7)--(6.3,2.3);
\draw (6.7,2.7)--(10,6);
\node [below] at (1,0) {$v_1$};
\node [below] at (2,0) {$v_2$};
\node [below] at (3,0) {$\cdot$};
\node [below] at (4,0) {$v_n$};
\node [below] at (5,0) {$\scriptstyle{\otimes}$};
\node [below right] at (5,0) {$w_1$};
\node [below right] at (6.5,0) {$w_2$};
\node [below right] at (8,0) {$\cdot$};
\node [below right] at (8.5,0) {$w_k$};
\node at (11,0) {.};
\end{tikzpicture}
\captionof{figure}{Pre-braiding extended to $T(V)$}\label{pic:BrTensor}
\end{center}
\end{remark}

\medskip
Recall further the famous {set inclusion} (not preserving the group structure)
\begin{align}
 S_n & \hookrightarrow B_n \notag\\
 s = \tau_{i_1}\tau_{i_2}\cdots \tau_{i_k}& \mapsto T_s := \sigma_{i_1}\sigma_{i_2}\cdots \sigma_{i_k} \label{eqn:T_s}
\end{align}
where 
\begin{itemize}
\item  $\tau_i \in S_n$ is the transposition of the neighboring elements $i$ and $i+1,$ 
\item  $\sigma_i$ is the corresponding generator of $B_n,$
\item  $\tau_{i_1}\tau_{i_2}\cdots \tau_{i_k}$ is one of the shortest words representing $s\in S_n.$
\end{itemize} 
 This inclusion factorizes through
$$S_n  \hookrightarrow B_n^+ \hookrightarrow B_n.$$

\medskip
The following subsets of symmetric groups deserve particular attention:
\begin{definition} 
 \emph{Shuffle sets} are the permutation sets
$$Sh_{p,q}:=\bigg\{ s \in S_{p+q} \text{~s.t.} 
\begin{array}{|l}
s(1)<s(2)<\ldots<s(p), \\
s(p+1)<s(p+2)<\ldots<s(p+q)
\end{array}
\bigg\}. $$
\end{definition}

In other words, one permutes $p+q$ elements preserving the order within the first $p$ and the last $q$ elements, just like when shuffling cards, which explains the name. Shuffles and their diverse modifications appear, sometimes quite unexpectedly, in various areas of mathematics.

\medskip
Everything is now ready for defining quantum shuffle algebras. 

\begin{definition}\label{def:qu_sh}
The \emph{quantum shuffle multiplication} on the tensor space $T(V)$ of a pre-braided vector space $(V,\sigma)$ is the $\k$-linear extension of the map
\begin{align}
 \sh = \sh^{p,q}:  V^{\otimes p}\otimes V^{\otimes q} & \longrightarrow V^{\otimes(p+q)},\notag \\
 \overline{v}\otimes\overline{w} & \longmapsto \overline{v}\sh\overline{w}:= \sum_{s\in Sh_{p,q}} T_s^\sigma (\overline{v}\overline{w}).\label{eqn:qu_sh}
\end{align}
Notation $T_s^\sigma$ stands for the lift $T_s \in B_n^+$ (cf. \eqref{eqn:T_s}) acting on $V^{\otimes n}$ via the pre-braiding $\sigma$ (cf. \eqref{eqn:^sigma}).

The algebra $\ShAlg_\sigma(V):=(T(V),\sh)$ is the \emph{quantum shuffle algebra} of $(V,\sigma).$ 
\end{definition}

By a {\emph{(pre-)braided Hopf algebra}} (in the sense of S.Majid, cf. Definition 2.2 in \cite{MajidBraided} for example) we mean an additional structure on a (pre-)braided vector space satisfying all the axioms of a Hopf algebra except for the compatibility between the multiplication and the comultiplication, which is replaced by the braided compatibility (this last notion is recalled in the following theorem). 

The quantum shuffle multiplication can be upgraded to an interesting pre-braided Hopf algebra structure:

\begin{theorem}\label{thm:Sh}
Let $(V,\sigma)$ be a pre-braided vector space.
\begin{enumerate}
\item  \label{prop:Ass}
The multiplication $\sh$ of $Sh_\sigma (V)$ is associative.
\item If $\sigma^2=\Id,$ then the multiplication $\sh$ is \textbf{$\ssigma$-commutative}, i.e.
$$\sh(\overline{v}\otimes\overline{w})=\sh(\ssigma(\overline{v}\otimes\overline{w}))$$
(with the extension $\ssigma$ of $\sigma$ to $T(V)$ from remark \ref{rmk:br_tensor}).
\item The element $1 \in \kk$ is a unit for $Sh_\sigma (V).$
\item The \textbf{deconcatenation} and the \textbf{augmentation} maps
\begin{align*}
 \Delta :  v_1 v_2 \ldots v_n & \longmapsto \sum_{p=0}^{n}v_1 v_2 \ldots v_p \otimes v_{p+1} \ldots v_n,
 & \varepsilon :  v_1 v_2 \ldots v_n & \longmapsto 0, \\
 1 & \longmapsto 1\otimes 1, &  1 & \longmapsto 1
\end{align*}
(where an empty product means $1$) define, after linearizing, a counital coalgebra structure on $T(V).$
\item These algebra and coalgebra structures are \textbf{$\ssigma$-compatible}, in the sense that
$$\Delta\circ\sh = (\sh\otimes\sh)\circ\ssigma_2\circ(\Delta\otimes\Delta).$$
\item \label{prop:antipode}
An antipode can be given on $Sh_\sigma (V)$ by linearizing the map
\begin{align}
 s : \overline{v} & \longmapsto (-1)^n T_{\Delta_n}^\sigma (\overline{v}), \qquad  \overline{v} \in V^{\otimes n},\notag\\
 1 & \longmapsto 1,\notag\\
 \text{where} \qquad\qquad \Delta_n &:= \bigl( \begin{smallmatrix}
 1 & 2 & \cdots & n\\ n & n-1 & \cdots & 1
\end{smallmatrix} \bigr)
\in S_n. \label{eqn:Garside}
\end{align}  
\end{enumerate}
The pre-braided vector space $(T(V),\ssigma)$ becomes thus a {pre-braided Hopf algebra}.
\end{theorem}

This result is well-known for invertible braidings (\cite{Rosso1}); we point out that it still holds when the pre-braiding admits no inverse. See \cite{Lebed} for a proof.

\begin{remark}
Dually (in the sense to be specified in subsection \ref{sec:co}), the tensor space of a pre-braided vector space $(V,\sigma)$ can be endowed with the \textit{\textbf{quantum co-shuffle comultiplication}}:
\begin{align}
\csh|_{V^{\otimes n}} &:= \sum_{p+q=n;\:p,q\ge 0} \csh^{p,q},\notag\\ 
\csh^{p,q} &:= \sum_{s\in Sh_{p,q}} T_{s^{-1}}^\sigma : V^{\otimes n} \longrightarrow V^{\otimes p}\otimes V^{\otimes q}, \label{eqn:cosh}
\end{align} 
which can be upgraded to a pre-braided Hopf algebra structure ``dual'' to that described in theorem \ref{thm:Sh}, denoted by $\ShCoalg_\sigma(V).$
\end{remark}

\section{Homologies of braided vector spaces}

In this section we introduce a homology theory of braided vector spaces. We propose two different viewpoints on our ``braided'' differentials (all the notions and properties are explained in this section):
\begin{center}
\begin{tabular}{|c|c|l|}
\hline 
 \textbf{approach} & \textbf{subsection} & \multicolumn{1}{c|}{\textbf{advantages}} \\
\hline
\textit{quantum co-shuffle} & \ref{sec:basic} & \ii the sign manipulation is hidden in \\
\textit{comultiplication} &  & the choice of the negative braiding $- \sigma,$ \\
and \textit{square zero}& &  \ii a subscript-free approach, \\
\textit{coelements} & &  \ii compact formulas;\\
\hline 
\textit{graphical calculus}:& \ref{sec:simplicial} & \ii a tool easy to manipulate, \\
 diagrams, braids & & \ii a finer structure of a pre-bisimplicial complex, \\ 
  & & \ii an intuitive definition of degenerate and  \\   
  & &   normalized complexes via braided coalgebras. \\ 
\hline 
\end{tabular}
   \captionof{table}{Two approaches to ``braided'' differentials}\label{tab:2Approches}
\end{center}

We also study certain homology operations for ``braided'' homologies, generalizing those introduced by M.Niebrzydowski and J.Przytycki in the context homologies of self-distributive structures (\cite{Prz2}, \cite{Prz1}), and we obtain a new interpretation and a generalization of J.-L.Loday's hyper-boundaries, automatically calculating all their compositions (cf. \cite{Cyclic}, exercise E.2.2.7).

Everything described here can be translated {verbatim} to the setting of $\kk$-modules. We prefer working with vector spaces for the sake of clarity. A categorical version of all the results is given in section \ref{sec:cat}.

\medskip

Fix a \underline{pre-braided $\k$-vector space} $(V,\sigma).$

\subsection{Pre-braiding + character $\longmapsto$ homology}\label{sec:basic}

We start with distinguishing elements of $V^*$ which behave with respect to the pre-braiding $\sigma$ as if it were just a flip $\tau$:

\begin{definition}\label{def:cuts}
Two co-elements $f,g \in V^*$ are called \emph{$\sigma$-compatible} if 
\begin{equation}\label{eqn:compatCochar}
(f \otimes g)\circ \sigma = g \otimes f, \qquad\text{and} \qquad (g \otimes f)\circ \sigma  = f \otimes g.
\end{equation}

 A \emph{braided character} is an $\epsilon \in V^*$ $\sigma$-compatible with itself, i.e.
\begin{equation}\label{eqn:Cut}
(\epsilon \otimes\epsilon)\circ \sigma  = \epsilon \otimes\epsilon,
\end{equation}
or, in the co-shuffle form, 
\begin{equation}\label{eqn:CutSh}
(\epsilon \otimes\epsilon)\circ \Csh^{1,1} = 0:V \otimes V \rightarrow \k.
\end{equation}
\end{definition}

We omit the part \emph{braided} of the last term when it does not lead to confusion.

The definition of braided character takes a simple graphical form:
 
\begin{center}
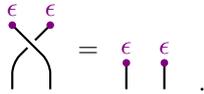

\begin{tikzpicture}[scale=0.5]
 \draw[thick,rounded corners] (0,-0.7) -- (0,0) -- (0.4,0.4);
 \draw[thick,rounded corners] (0.6,0.6) -- (1,1);
 \draw[thick,rounded corners] (1,-0.7) -- (1,0) -- (0,1);
 \fill[violet] (0,1) circle (0.1);
 \fill[violet] (1,1) circle (0.1);
 \node at (0,1) [above,violet] {$\epsilon$};
 \node at (1,1) [above,violet] {$\epsilon$};
 \node at (2,0.3)  {$=$};
 \draw[thick] (3,0) -- (3,-0.7);
 \draw[thick] (4,0) -- (4,-0.7);
 \node at (3,0) [above,,violet] {$\epsilon$};
 \node at (4,0) [above,,violet] {$\epsilon$}; 
 \fill[violet] (4,0) circle (0.1);
 \fill[violet] (3,0) circle (0.1);
 \node at (5,-0.7) {.}; 
\end{tikzpicture}
\captionof{figure}{Braided character}\label{pic:BrCoChar}
\end{center}

The labels $\epsilon$ are often omitted when clear from the context.
 
 In section \ref{sec:examples}, we recover familiar notions of characters for algebraic structures (such as associative algebras) as examples of braided characters for the corresponding ``structural'' braidings, which justifies our term. Counits often turn out to be braided characters, hence the notation $\epsilon.$

\begin{remark}\label{rmk:CharHomToTrivial}
Braided characters can also be regarded as {\emph{braided morphisms}} $\epsilon : V \rightarrow \k,$ where $\k$ is endowed with the {\emph{trivial braiding}}
$\k \otimes \k \simeq \k \overset{\Id_\k}{\longrightarrow} \k \simeq \k \otimes \k.$
This is consistent with the interpretation of usual characters for algebraic structures as homomorphisms to trivial structures.
\end{remark}

\medskip

A pre-braiding and a braided character are sufficient for constructing homologies:

\begin{theorem}\label{thm:cuts}
Let $(V,\sigma)$ be a pre-braided vector space. 
\begin{enumerate}
\item 
For a braided character $\epsilon,$ the maps
\begin{align*}
V^{\otimes n} & \longrightarrow V^{\otimes (n-1)},\\
{^\epsilon}\! d : \overline{v} & \longmapsto \epsilon_1 \circ \Csh^{1,n-1}(\overline{v}), \\
d^\epsilon : \overline{v} & \longmapsto (-1)^{n-1} \epsilon_n \circ\Csh^{n-1,1}(\overline{v}) 
\end{align*}
(cf. notation \eqref{eqn:phi_i}) define a bidifferential on $T(V).$

\item For two braided characters $\epsilon$ and $\zeta,$ one gets a differential bicomplex 
 $(T(V),{^{\epsilon}}\! d,d^{\zeta}).$ If the braided characters are moreover $\sigma$-compatible, then one gets differential bicomplexes $(T(V),{^{\epsilon}}\!d,{^{\zeta}}\! d)$ and $(T(V),d^\epsilon,d^\zeta).$ 
 \end{enumerate}
\end{theorem}

\begin{proof}
The verifications use the coassociativity of $\Csh$ and the defining property \eqref{eqn:CutSh} of braided characters. For example,
\begin{align*}
{^{\epsilon}}\! d \circ {^{\epsilon}}\! d (\overline{v}) &= \epsilon_1 \circ(\epsilon \otimes \Csh^{1,n-2}) \circ \Csh^{1,n-1} (\overline{v}) \\
& = ((\epsilon \otimes \epsilon)\circ \Csh^{1,1})_1 \circ \Csh^{2,n-2} (\overline{v}) \\
& = 0 \circ \Csh^{2,n-2} (\overline{v})= 0,
\end{align*}
and, for compatible braided characters,
\begin{align*}
({^{\zeta}}\! d \circ {^{\epsilon}}\! d +{^{\epsilon}}\! d \circ {^{\zeta}}\! d) (\overline{v}) &= \zeta_1 \circ(\epsilon \otimes \Csh^{1,n-2}) \circ \Csh^{1,n-1} (\overline{v}) +\epsilon_1 \circ(\zeta \otimes \Csh^{1,n-2}) \circ \Csh^{1,n-1} (\overline{v})\\
& = ((\zeta \otimes \epsilon + \epsilon \otimes \zeta)\circ \Csh^{1,1})_1 \circ \Csh^{2,n-2} (\overline{v}),\end{align*}
which, due to \eqref{eqn:compatCochar}, equals
$$((\zeta \otimes \epsilon - \epsilon \otimes \zeta) + (\epsilon \otimes \zeta-\zeta \otimes \epsilon ))_1 \circ \Csh^{2,n-2} (\overline{v}) =  0. \qedhere$$
\end{proof}

This proof can be understood as follows: the multiplication by a square zero element in the pre-braided Hopf algebra dual to $\ShCoalg_{-\sigma}$ is a square zero operator.  

The theorem gives for two braidings two compatible differentials on $T(V).$ Their linear combinations are then also differentials;  ${^\epsilon}\! d - d^\epsilon$ is a reccurrent example in practice. All such (bi)differentials, corresponding (bi)complexes and homologies are called \textit{\textbf{braided}} in what follows.

\medskip

We now give a survey of existing ``braided'' homologies, comparing them to our theory.

\begin{enumerate}
\item 
In \cite{HomologyYB}, J.S.Carter, M.Elhamdadi and M.Saito develop a homology theory for solutions $(S,\sigma)$ of the set-theoretic YBE using combinatorial and geometric methods completely different from ours. They also provide applications to virtual knot invariants. Their differential on $(\ZZ S)^{\otimes n}$ is our braided differential ${^\varepsilon}\!d - d{^\varepsilon},$ where $\varepsilon$ is the linearization of the map $\varepsilon:S \rightarrow \ZZ$ assigning $1$ to any $a \in S.$
\item
Our applications ${^\varepsilon}\!d,$ where $\varepsilon\in V^*$ are not necessarily braided characters, also recover the \emph{braided-differential calculus} of S.Majid (\cite{BraidedDiff}). He defines an addition law (related to the quantum co-shuffle comultiplication) on the quantum plane associated to a braiding, and defines a differentiation as an infinitesimal translation. In particular, taking 
\begin{itemize}
\item one-variable polynomials $T(V) = \k [x]$ (i.e. $V = \k x$),
\item the opposite of the $q$-flip $x \otimes x \mapsto  q x \otimes x$ (with $q \in \k^*$) as a braiding,
\item and the linearization of the map $\varepsilon(x)=1,$
\end{itemize}
 one gets the famous $q$-differentials
 $${^\varepsilon}\!d (x^{\otimes n}) = (n)_q x^{\otimes (n-1)}, \qquad (n)_q:=\tfrac{q^{n}-1}{q-1} = q^{n-1} + \cdots + q + 1.$$
\item The last approach to ``braided'' cohomologies to be mentioned here is M.Eisermann's \emph{Yang-Baxter cochain complex}, cf. \cite{Eisermann}. Motivated by the study of deformations of Yang-Baxter operators, he defines a degree $1$ differential on $\Hom_\k(V^{\otimes n},V^{\otimes n}).$ His second cohomology groups classify infinitesimal Yang-Baxter deformations. We do not know precisely how his construction is related to ours, but the parallels between the graphical versions of the two are very suggestive.
\end{enumerate}

\subsection{Comultiplication $\longmapsto$ degeneracies}\label{sec:simplicial}

In a more detailed study of the structure of braided (bi)complexes from theorem \ref{thm:cuts}, the simplicial approach proves to be particularly helpful.

 First, recall the notion of simplicial vector spaces (cf. \cite{Cyclic} for details and \cite{Prz1} for weak simplicial notions; note that our definition is a shifted version of theirs, and that our definition of bisimplicial vector spaces is different from the one in \cite{Cyclic}): 

\begin{definition}\label{def:Simpl}
Consider a collection of $\k$-vector spaces $V_n, \: n \ge 0,$
equipped with linear maps $d_{n;i}:V_n\rightarrow V_{n-1}$ (and $d'_{n;i}:V_n\rightarrow V_{n-1}$ and/or $s_{n;i}:V_n\rightarrow V_{n+1}$ when necessary) with $1 \le i \le n,$ denoted simply by $d_i,d'_i,s_i$ when the subscript $n$ is clear from the context. This datum is called
\begin{itemize}
\item \emph{a presimplicial vector space} if
\begin{align}
d_{i} d_{j} &= d_{j-1} d_{i} & \forall 1\le i < j \le n;\label{eqn:simpl1}
\end{align}
\item \emph{a very weakly simplicial vector space} if moreover
\begin{align}
s_{i} s_{j} &= s_{j+1} s_{i} & \forall 1\le i \le j \le n,\label{eqn:simpl2}\\
d_{i} s_{j} &= s_{j-1} d_{i} & \forall 1\le i < j \le n,\label{eqn:simpl3}\\
d_{i} s_{j} &= s_{j} d_{i-1} & \forall 1\le j+1 < i \le n;\label{eqn:simpl4}
\end{align}
\item \emph{a weakly simplicial vector space} if moreover
\begin{align}
d_{i} s_{i} &= d_{i+1} s_{i} & \forall 1\le i \le n;\label{eqn:simpl5}
\end{align}
\item \emph{a simplicial vector space} if moreover
\begin{align}
d_{i} s_{i} &= \Id_{V_n}  & \forall 1\le i \le n;\label{eqn:simpl6}
\end{align}
\item \emph{a pre-bisimplicial vector space} if \eqref{eqn:simpl1} holds for the $d_i$'s, the $d'_i$'s and their mixture:
\begin{align}
d_{i} d'_{j} &= d'_{j-1} d_{i}, \qquad d'_{i} d_{j} = d_{j-1} d'_{i} & \forall 1\le i < j \le n;\label{eqn:simpl1'}
\end{align}
\item \emph{a (weakly / very weakly) bisimplicial vector space} if it is pre-bisimplicial, with (weakly / very weakly) simplicial structures $(V_n,d_{n;i},s_{n;i})$ and $(V_n,d'_{n;i},s_{n;i}).$ 
\end{itemize}
The omitted subscripts $n,n \pm 1$ are those which guarantee that the source of all the above mentioned morphisms is $V_n.$ The $d_i$'s and the $s_i$'s are called \emph{face} (resp. \emph{degeneracy}) maps.
\end{definition}

Simplicial vector spaces are interesting because of the following properties (see \cite{Cyclic} for most proofs):
\begin{proposition}\label{thm:SimplBasics}
\begin{enumerate}
\item\label{it:totaldiff} For any presimplicial vector space $(V_n,d_{n;i}),$ the map 
$\partial_n:=\sum_{i=1}^n (-1)^{i-1} d_{n;i}$
is a differential (called the \textbf{total differential}) for the graded vector space  $\wV:=\bigoplus_{n\ge 0} V_n.$
\item For any pre-bisimplicial vector space $(V_n,d_{n;i},d'_{n;i}),$ there is a bidifferential structure on $\wV$ given by $\partial_n$ and 
$\partial'_n:=\sum_{i=1}^n (-1)^{i-1} d'_{n;i}.$
\item\label{it:degen} For any weakly simplicial vector space $(V_n,d_{n;i},s_{n;i}),$ the complex $(V_n,\partial_n)$ contains a subcomplex (called \textbf{the degenerate subcomplex}) $D_n:=\sum_{i=1}^{n-1} s_{n-1;i}(V_{n-1}).$
\item If our vector space turns out to be simplicial, then the degenerate subcomplex is acyclic, hence $V_*$ is quasi-isomorphic to the \textbf{normalized complex} $N_*:=V_*/D_*.$
\item In the weakly bisimplicial case, $D_*$ is a sub-bicomplex of $V_*,$ acyclic in the bisimplicial setting.
\end{enumerate}
\end{proposition}

\medskip

We will soon show that the bicomplexes from theorem \ref{thm:cuts} come from pre-bisimplicial structures. As for degeneracies, they arise from the following structure:

\begin{definition}\label{def:BraidedCoalg}
A pre-braided vector space $(V,\sigma)$ endowed with a comultiplication $\Delta:V\rightarrow V\otimes V$ is called \emph{a pre-braided coalgebra} if
\begin{itemize}
\item $\Delta$ is \emph{co-associative}:
\begin{equation}\label{eqn:CoAss}
(\Delta \otimes \Id_V) \circ \Delta = (\Id_V \otimes \Delta) \circ \Delta: V \rightarrow V\otimes V\otimes V,
\end{equation}
\item and $\Delta$ is \emph{compatible} with the pre-braiding:
\begin{align}
\Delta_2\circ\sigma &= \sigma_1 \circ \sigma_2 \circ\Delta_1:V^{\otimes 2} \rightarrow V^{\otimes 3},\label{eqn:BrCoalg}\\
\Delta_1\circ\sigma &= \sigma_2 \circ \sigma_1 \circ\Delta_2:V^{\otimes 2} \rightarrow V^{\otimes 3}.\label{eqn:BrCoalg'}
\end{align}
\end{itemize}

One talks about \emph{semi-pre-braided coalgebras} if only \eqref{eqn:BrCoalg} holds.

A (semi-)pre-braided coalgebra is called \emph{$\sigma$-cocommutative} if
\begin{equation}\label{eqn:sigma_cocomm}
\sigma\circ\Delta = \Delta:V \rightarrow V \otimes V.
\end{equation}
\end{definition}

The properties from the definition are graphically depicted as 
\begin{center}
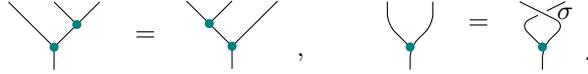

\begin{tikzpicture}[scale=0.3]
\draw (0,0)--(2,-2);
\draw (4,0)--(2,-2);
\draw (3,-1)--(2,0);
\draw (2,-3)--(2,-2);
\fill[teal] (3,-1) circle (0.2);
\fill[teal] (2,-2) circle (0.2);
\node at (6,-1.5) {$=$};
\node at (7,-1.5) {};
\end{tikzpicture}
\begin{tikzpicture}[scale=0.3]
\draw (0,0)--(2,-2);
\draw (4,0)--(2,-2);
\draw (1,-1)--(2,0);
\draw (2,-3)--(2,-2);
\fill[teal] (1,-1) circle (0.2);
\fill[teal] (2,-2) circle (0.2);
\node at (5,-2.5) {,};
\node at (8,-1.5) {};
\end{tikzpicture}
\begin{tikzpicture}[scale=0.3]
 \draw[rounded corners] (0,-1) -- (0,0)-- (1,1)--(1,2);
 \draw[rounded corners] (0,0)-- (-1,1)--(-1,2);
 \fill[teal] (0,0) circle (0.2);
 \node at (3,1) {$=$};
 \node at (4,1) {};
\end{tikzpicture} \begin{tikzpicture}[scale=0.3]
 \draw[rounded corners] (0,-1) -- (0,0)-- (1,1)--(-1,2);
 \draw[rounded corners] (0,0)-- (-1,1)--(-0.2,1.4);
 \draw[rounded corners] (0.2,1.6)-- (1,2);
 \fill[teal] (0,0) circle (0.2);
 \node at (1,1.5) {$\sigma$};
 \node at (2,-0.5) {.};
\end{tikzpicture}
   \captionof{figure}{Coassociativity and $\sigma$-cocommutativity}\label{pic:Coass}
\end{center}
  \begin{center}
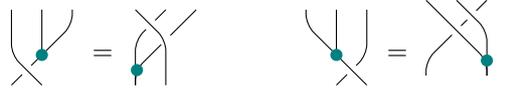

\begin{tikzpicture}[scale=0.4]
 \draw[rounded corners](0.65,0.65) -- (2,2)-- (2,2.5);
 \draw (0,0) -- (0.35,0.35);
 \draw (1,1) -- (1,2.5);
 \draw[rounded corners](1,0) -- (0,1)-- (0,2.5);
 \fill[teal] (1,1) circle (0.2);
 \node at (3,1) {$=$};
\end{tikzpicture}
\begin{tikzpicture}[scale=0.4]
 \draw[rounded corners](0,-0.5) -- (0,0)--(0.85,0.85);
 \draw (1.15,1.15) --   (2,2);
 \draw[rounded corners](0,-0.5) -- (0,1)--(0.35,1.35);
 \draw (0.65,1.65) --  (1,2);
 \draw[rounded corners] (1,-0.5) -- (1,1) -- (0,2);
 \fill[teal] (0.05,0) circle (0.2);
\end{tikzpicture}
\begin{tikzpicture}[scale=0.4]
 \node at (-3,1) {};
 \draw[rounded corners](1.65,0.65) -- (2,1)-- (2,2.5);
 \draw (1,0) -- (1.35,0.35);
 \draw (1,1) -- (1,2.5);
 \draw[rounded corners](2,0) -- (0,2)-- (0,2.5);
 \fill[teal] (1,1) circle (0.2);
 \node at (3,1) {$=$};
\end{tikzpicture}
\begin{tikzpicture}[scale=0.4]
 \draw[rounded corners](0,-0.5) -- (0,0)-- (0.85,0.85);
 \draw (1.15,1.15) -- (1.35,1.35);
 \draw (1.65,1.65) -- (2,2);
 \draw[rounded corners](2,-0.5) -- (2,1)-- (1,2);
 \draw[rounded corners](2,-0.5) -- (2,0)-- (0,2);
 \fill[teal] (2,0) circle (0.2);
 \node at (3,-0.5) {.};
\end{tikzpicture}
\captionof{figure}{Braided coalgebra}\label{pic:BrCoalg}
 \end{center}
Here the comultiplication $\Delta$ is represented as \begin{tikzpicture}[scale=0.3]
 \draw(0,-1) -- (0,0)-- (1,1);
 \draw (0,0)-- (-1,1);
 \fill[teal] (0,0) circle (0.2);
\end{tikzpicture}.

Theorem \ref{thm:cuts} admits now the following simplicial interpretation:

\begin{theorem}\label{thm:BraidedSimplHom}
Let $(V,\sigma)$ be a pre-braided vector space. 

\begin{enumerate}

\item For braided characters $\epsilon$ and $\zeta,$ the maps
\begin{align}
d_{n;i}(\overline{v})&:= \epsilon_1 \circ T_{p_{i,n}}^\sigma(\overline{v}),\label{eqn:braided_di}\\
d'_{n;i}(\overline{v})&:= \zeta_n \circ T_{p'_{i,n}}^\sigma(\overline{v}),\label{eqn:braided_ddi}
\end{align}
define a pre-bisimplicial structure on the tensor vector space $T(V).$ 
Here $p_{i,n} \in S_n$ (resp. $p'_{i,n} \in S_n$) is the permutation moving the $i$th element to the leftmost (resp. rightmost) position, and the notation $T_s^\sigma$ from \eqref{eqn:T_s} and \eqref{eqn:^sigma} is used.

\item The total differentials $\partial$ and $\partial'$ coincide with the ``shuffle'' differentials ${^\epsilon}\! d$ and, respectively, $d^{\zeta}$ from theorem \ref{thm:cuts}.

\item If the braided characters are moreover $\sigma$-compatible, then the $d_i$'s for $\epsilon$ and the $d_i$'s for $\zeta$ define a pre-bisimplicial structure on $T(V).$

\item If for a comultiplication $\Delta$ the triple $(V,\sigma,\Delta)$ is a pre-braided coalgebra, then the maps
\begin{equation}\label{eqn:braided_si}
s_{n;i} := \Delta_i
\end{equation}
complete the preceding structures into very weakly bisimplicial ones.

\item If $(V,\sigma,\Delta)$ is a semi-pre-braided coalgebra, then the data $(V^{\otimes n},d_{n;i},s_{n;i})$ described above give a very weakly simplicial vector space only.

\item If $\Delta$ is moreover {$\sigma$-cocommutative}, then the above structures are weakly (bi)simplicial.
\end{enumerate}

\end{theorem}

The face and degeneracy maps from the theorem are graphically depicted as 
\begin{center}
\begin{tikzpicture}[scale=0.3]
 \draw (0,1) --(0,-1.6);
 \draw (0,-2.4) --  (0,-7);
 \node at (-4,-3) {$d_{n;i}=$};
 \draw[thick,red,rounded corners] (-1,0) -- (-1,-1) -- (4,-6) -- (4,-7);
 \fill[violet] (-1,0) circle (0.2);
 \node at (-1,0)  [above, violet] {$\epsilon$};
 \draw (3,1) -- (3,-4.6);
 \draw (3,-5.4) -- (3,-7);
 \draw (5,1) -- (5,-7);
 \draw (8,1) -- (8,-7);
 \node at (0.3,-2.3) [ above right] {$\sigma$};
 \node at (3.3,-5.3) [above right] {$\sigma$};
 \node at (0,-7) [below] {$\scriptstyle 1$};
 \node at (1.5,-7) [below] {$\scriptstyle\ldots$};
 \node at (3,-7) [below] {${\scriptstyle i-1}$}; 
 \node at (4,-7) [below] {$\scriptstyle i$};
 \node at (5,-7) [below] {$\scriptstyle i+1$};
 \node at (6.5,-7) [below] {$\scriptstyle \ldots$};
 \node at (8,-7) [below] {$\scriptstyle{n}$}; 
\end{tikzpicture}
\begin{tikzpicture}[scale=0.3]
 \draw (0,1) -- (0,-7);
 \node at (-9,-3) {};
 \node at (-4,-3) {$d'_{n;i}=$};
 \draw[thick,red,rounded corners] (9,0) -- (9,-1) -- (8.3,-1.7);
 \draw[thick,red,rounded corners] (7.7,-2.3) -- (5.3,-4.7);
 \draw[thick,red,rounded corners] (4.7,-5.3)--(4,-6) -- (4,-7);
 \fill[violet] (9,0) circle (0.2);
 \node at (9,0) [above,violet] {$\zeta$};
 \draw (3,1) -- (3,-7);
 \draw (5,1) -- (5,-7);
 \draw (8,1) -- (8,-7);
 \node at (5,-5) [right] {$\sigma$};
 \node at (8,-2) [right] {$\sigma$};
 \node at (0,-7) [below] {$\scriptstyle 1$};
 \node at (1.5,-7) [below] {$\scriptstyle \ldots$};
 \node at (3,-7) [below] {$\scriptstyle{i-1}$}; 
 \node at (4,-7) [below] {$\scriptstyle i$};
 \node at (5,-7) [below] {$\scriptstyle i+1$};
 \node at (6.5,-7) [below] {$\scriptstyle\ldots$};
 \node at (8,-7) [below] {$\scriptstyle{n}$}; 
\end{tikzpicture}
\end{center}
\begin{center}
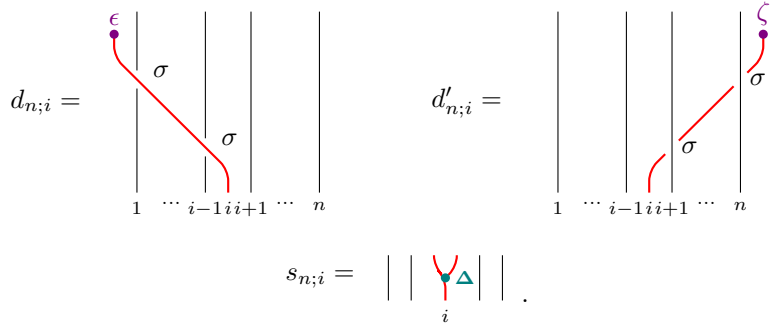

\begin{tikzpicture}[scale=0.3]
 \node at (-7,1) {};
 \node at (-4,1) {$s_{n;i}=$};
 \draw (-1,0) -- (-1,2);
 \draw (0,0) -- (0,2);
 \draw (4,0) -- (4,2);
 \draw (3,0) -- (3,2);
 \draw[thick,red,rounded corners] (1.5,0) -- (1.5,1) -- (1,1.5) -- (1,2);
 \draw[thick,red,rounded corners] (1.5,1) -- (2,1.5) -- (2,2);
 \fill[teal] (1.5,1) circle (0.2);
 \node at (1.5,1) [right,teal] {$\scriptstyle \mathbf{\Delta}$};
 \node at (1.5,0) [below] {$\scriptstyle i$};
 \node at (5,0) {.};
\end{tikzpicture}
   \captionof{figure}{Simplicial structures}\label{pic:Simpl}
\end{center}

\begin{proof}
One has to deduce the ``simplicial'' relations (definition \ref{def:Simpl}) from the properties of the structures on $V,$ which were conceived precisely for these relations to hold. This can be done graphically, using the pictorial interpretation of face and degeneracy maps, and the graphical definitions of a ($\sigma$-cocommutaive) pre-braided coalgebra presented above, as well as the pictorial versions of the YBE (fig. \ref{pic:YB}) and of the definition of braided characters (fig. \ref{pic:BrCoChar}).

For instance,
\begin{center}
\begin{tikzpicture}[scale=0.3]
 \draw (0,0) --(0,3.7);
 \draw (0,4.3) --  (0,4.7);
 \draw (0,5.3) --  (0,6);
 \draw[thick,red,rounded corners] (1,0) -- (1,2.6);
 \draw[thick,red,rounded corners] (1,3.4) -- (1,4) -- (-0.5,5.5);
 \fill[violet] (-0.5,5.5) circle (0.2);
 \draw (2,0) --(2,1.5);
 \draw (2,2.5) -- (2,6);
 \draw[thick,red,rounded corners] (3,0) -- (3,1) -- (-1.5,5.5);
 \fill[violet] (-1.5,5.5) circle (0.2);
 \draw (4,0) -- (4,6);
 \node at (1,0) [below] {$\scriptstyle i$};
 \node at (3,0) [below] {$\scriptstyle j$};
 \node at (5.5,3) {$\stackrel{(1)}{=}$};
\end{tikzpicture}
\begin{tikzpicture}[scale=0.3]
 \draw (0,0) --(0,1.5);
 \draw (0,2.5) --  (0,3.5);
 \draw (0,4.5) --  (0,6);
 \draw[thick,red,rounded corners] (1,0) -- (1,1) -- (-0.5,2.5) -- (-0.5,4.1);
 \draw[thick,red,rounded corners] (-0.5,4.8) -- (-0.5,5.5);
 \fill[violet] (-0.5,5.5) circle (0.2);
 \draw (2,0) --(2,1.5);
 \draw (2,2.5) -- (2,6);
 \draw[thick,red,rounded corners] (3,0) -- (3,1) -- (-1,5) -- (-1,5.5);
 \fill[violet] (-1,5.5) circle (0.2);
 \draw (4,0) -- (4,6);
 \node at (1,0) [below] {$\scriptstyle i$};
 \node at (3,0) [below] {$\scriptstyle j$};
 \node at (5.5,3) {$\stackrel{(2)}{=}$};
\end{tikzpicture}
\begin{tikzpicture}[scale=0.3]
 \draw (0,0) --(0,1.5);
 \draw (0,2.5) --  (0,3.5);
 \draw (0,4.5) --  (0,6);
 \draw[thick,red,rounded corners] (1,0) -- (1,1) -- (-0.5,2.5) -- (-0.5,3);
 \fill[violet] (-0.5,3) circle (0.2);
 \draw (2,0) --(2,1.5);
 \draw (2,2.5) -- (2,6);
 \draw[thick,red,rounded corners] (3,0) -- (3,1) -- (-0.5,4.5) -- (-0.5,5);
 \fill[violet] (-0.5,5) circle (0.2);
 \draw (4,0) -- (4,6);
 \node at (1,0) [below] {$\scriptstyle i$};
 \node at (3,0) [below] {$\scriptstyle j$};
 \node at (5,0) {.};
\end{tikzpicture}
   
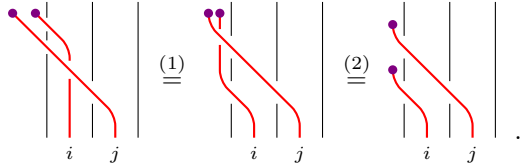
\captionof{figure}{Graphical proof of $d_{i} d_{j} = d_{j-1} d_{i} \; \forall 1\le i < j \le n.$} \label{pic:GraphProofSimpl}
\end{center}
Here

(1) is a repeated application of YBE;

(2) follows from the definition \eqref{eqn:Cut} of a braided character (cf. fig. \ref{pic:BrCoChar}). 
\end{proof}

\begin{remark}
When checking the axioms of different types of simplicial structures in the proof, one can get rid of the tiresome index chasing by reasoning in terms of strands. For example, pulling a strand to the left commutes with applying the branching $\Delta$ to any other strand if a strand can pass over a branching.
\end{remark}

\subsection{Arrow and concatenation operations}\label{sec:HomOper}

The last face map $d_{n+1;n+1}$ on $V^{\otimes (n+1)}$ is of particular interest. It inspires the definition of a useful operation on $T(V)$:
\begin{definition}\label{def:ArrowOp}
Take a braided character $\epsilon$ on $(V,\sigma).$ For a $w \in V,$ we call an \emph{arrow operation} on $T(V)$ the map defined by
$${^\epsilon}\!\pi_w(\overline{v}):= d_{n+1;n+1}(\overline{v}w) =\epsilon_{1} \circ T_{p_{n+1,n+1}}^\sigma(\overline{v}w), \qquad \forall \overline{v} \in V^{\otimes n}.$$
\end{definition}

The notation and the name come from the graphical presentation:
\begin{center}
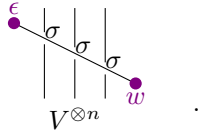

\begin{tikzpicture}[scale=0.4]
 \draw (0,0) -- (0,1.8);
 \draw (0,2.2) -- (0,3);
 \draw (1,0) -- (1,1.3);
 \draw (1,1.7) -- (1,3);
 \draw (2,0) -- (2,0.8);
 \draw (2,1.2) -- (2,3);
 \draw (3,0.5) -- (-1,2.5);
 \node at (-0.3,1.7) [above right]{$\sigma$};
 \node at (0.7,1.2) [above right]{$\sigma$};
 \node at (1.7,0.7) [above right]{$\sigma$};
 \fill[violet] (-1,2.5) circle (0.2);
 \fill[violet] (3,0.5) circle (0.2);
 \node at (3,0.5) [below,violet] {$w$};
 \node at (1,0) [below] {$V^{\otimes n}$};
 \node at (5,0) [below] {.};
 \node at (-1,2.5) [above, violet] {$\epsilon$};
\end{tikzpicture}
   \captionof{figure}{Arrow operation}\label{pic:NWArrow}
\end{center}

This map will be interpreted in terms of modules over pre-braided vector spaces and adjoint maps in proposition \ref{thm:adjoint_module}. Here we study its properties, and get a generalization of some homology operations from \cite{Prz2} and \cite{Prz1}. Our constructions are deeply inspired by those papers. 

 Start with some technical definitions:

\begin{definition}\label{def:NormCompat}
Take a pre-braided vector space $(V,\sigma).$

\begin{itemize}
\item
A \emph{normalized pair} is an element $w \in V$ and a co-element $\psi\in V^*$ satisfying $ \psi(w)=1.$
\item A $w\in V$ and a $\psi\in V^*$ are called \emph{right $\sigma$-compatible} if 
\begin{equation}\label{eqn:RightSigmaCompat} 
 (\Id_V \otimes \psi)\circ\sigma\circ (v \otimes w) = \psi(v)w \qquad \forall v \in V.
\end{equation}

\item The pre-braiding $\sigma$ is called \emph{natural} with respect to a $w\in V$ if
\begin{align}
 \sigma\circ (w \otimes v) &= v \otimes w \qquad \forall v \in V,\label{eqn:NatWRTSigma}\\
 \sigma\circ (v \otimes w) &= w \otimes v \qquad \forall v \in V,\label{eqn:NatWRTSigma'}
\end{align}
and \emph{semi-natural} (or \emph{demi-natural}) if only \eqref{eqn:NatWRTSigma} (resp. \eqref{eqn:NatWRTSigma'}) holds.
\end{itemize}
\end{definition}

The properties from the definition graphically mean
\begin{center}
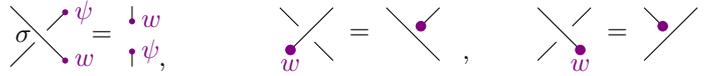

\begin{tikzpicture}[scale=0.4]
 \draw (0,0) -- (0.8,0.8);
 \draw (1.2,1.2) -- (1.8,1.8);
 \draw (1.8,0.2) -- (0,2);
 \node at (1,1) [left]{$\sigma$};
 \fill[violet] (1.8,1.8) circle (0.1);
 \fill[violet] (1.8,0.2) circle (0.1);
 \node at (1.8,0.2) [right, violet] {$w$};
 \node at (1.8,1.8) [right, violet] {$\psi$};
 \node at (3,1) {$=$};
 \draw (4,0) -- (4,0.5);
 \draw (4,2) -- (4,1.5);
 \fill[violet] (4,0.5) circle (0.1);
 \fill[violet] (4,1.5) circle (0.1);
 \node at (4,1.5) [right, violet] {$w$};
 \node at (4,0.5) [right, violet] {$\psi$};
 \node at (5,0){,};
 \node at (8,0){};
\end{tikzpicture}
\begin{tikzpicture}[scale=0.7]
 \draw (1,0) -- (0.65,0.35); 
 \draw (0.35,0.65) -- (0,1); 
 \draw (0.2,0.2) -- (1,1); 
 \fill [violet] (0.2,0.2) circle (0.1);
 \node at (0.2,0.2) [below,violet] {$w$};
 \node at (1.5,0.5) {$=$};
 \draw (3,0) -- (2,1); 
 \draw (2.65,0.65) -- (3,1); 
 \fill [violet] (2.65,0.65) circle (0.1);
 \node at (3.5,0) {,};
 \node at (4.5,0){};
\end{tikzpicture}
\begin{tikzpicture}[scale=0.7]
 \draw (0,0) -- (0.35,0.35); 
 \draw (0.65,0.65) -- (1,1); 
 \draw (0.8,0.2) -- (0,1); 
 \fill [violet] (0.8,0.2) circle (0.1);
 \node at (0.8,0.2) [below,violet] {$w$};
 \node at (1.5,0.5) {$=$};
 \draw (2,0) -- (3,1); 
 \draw (2.35,0.65) -- (2,1); 
 \fill [violet] (2.35,0.65) circle (0.1);
 \node at (3.5,0) {.};
\end{tikzpicture}
   \captionof{figure}{Right $\sigma$-compatibility and naturality with respect to an element}
   \label{pic:NatElement}
\end{center}

The naturality can be interpreted as follows: the element $w$ can ``pass through a crossing'' to the left / to the right. 

Note that condition \eqref{eqn:NatWRTSigma'} implies \eqref{eqn:RightSigmaCompat} for any $\psi.$

The compatibility of arrow operations ${^\epsilon}\!\pi_w$ with the braided differential ${^{\epsilon}}\! d$ from theorem \ref{thm:cuts} is a consequence of theorem \ref{thm:BraidedSimplHom}, where we interpret ${^{\epsilon}}\! d$ as a total differential for which ${^\epsilon}\!\pi$ is the last face map. This is an inspiring remark for the following analysis of the behavior of our braided differentials with respect to arrow operations and \textbf{\textit{concatenation operations}}
$\ov \mapsto \ov w$ on $T(V),$ for a fixed $w \in V.$ 

\begin{proposition}\label{thm:hom_operations}
Let $(V,\sigma)$ be a pre-braided vector space with braided characters $\epsilon,\xi,\zeta,$ the first two being $\sigma$-compatible, and the last one being right $\sigma$-compatible with an element $w\in V.$ 
\begin{enumerate}
\item\label{item:HomOperators} The map ${^\epsilon}\!\pi_w$ is a bicomplex map for $(T(V),{^{\xi}}\! d, d^{\zeta}),$ i.e.
\begin{align*}
{^{\xi}}\! d \circ {^\epsilon}\!\pi_w (\ov) =& {^\epsilon}\!\pi_w\circ {^{\xi}}\! d(\ov),\\
d^{\zeta} \circ {^\epsilon}\!\pi_w (\ov) =& {^\epsilon}\!\pi_w \circ d^{\zeta}(\ov).
\end{align*}
\item\label{item:NWarrow} The following relations hold between the concatenation operations and the braided differentials:
\begin{align*}
{^{\epsilon}}\! d(\ov w) &= {^{\epsilon}}\! d(\ov) w + (-1)^n {^\epsilon}\!\pi_w (\ov),\\
d^{\zeta}(\ov w) &= d^{\zeta}(\ov) w  + (-1)^n \zeta(w) \ov.
\end{align*}
\item\label{item:NWarrowDemiNat} If $\sigma$ is demi-natural with respect to $w,$ then
$$ {^\epsilon}\!\pi_w = \epsilon(w) \Id_{T(V)}.$$
\end{enumerate}
Here the notation $\ov$ stays for any pure tensor in $V^{\otimes n}.$
\end{proposition}

This result admits an evident ``left'' version (with respect to $w$).

\begin{proof}

Point \ref*{item:HomOperators} follows, in the same way as the proof of theorem \ref{thm:BraidedSimplHom}, from the YBE for $\sigma$ and from the  $\sigma$-compatibilities (use for instance the graphical caclulus).

Point \ref*{item:NWarrow} can be checked using the pre-braided Hopf algebra structure on $\ShCoalg_{-\sigma}(V).$  For instance, for the left differentials one has
\begin{align*}
{^{\epsilon}}\! d(\ov w) & = (\epsilon \otimes \Id_n)\circ \Csh^{1,n} (\ov w)  \\
&\overset{(*)}{=}  (\epsilon \otimes \Id_n)  (\Csh^{1,n-1} (\ov) w + T_{p_{n+1,n+1}}^{-\sigma} (\ov w) )\\
& =  (\epsilon \otimes \Id_n)  (\Csh^{1,n-1} (\ov) w) + (-1)^n(\epsilon \otimes \Id_n) \circ T_{p_{n+1,n+1}}^{\sigma} (\ov w) \\
& = {^{\epsilon}}\! d(\ov) w + (-1)^n {^\epsilon}\!\pi_w(\ov).
\end{align*}
Equality $(*)$ is the compatibility between the multiplication and the comultiplication in $\ShCoalg_{-\sigma}(V).$

Point \ref*{item:NWarrowDemiNat} is straightforward.
\end{proof}

\begin{corollary}\label{crl:hom_operations}
\begin{enumerate}
\item In the settings of the previous proposition, the arrow operation ${^\epsilon}\!\pi_w$ is homotopic to zero on the complex $(T(V),{^{\epsilon}}\! d),$ and to $\zeta(w) \Id_{T(V)}$ on $(T(V),{^{\epsilon}}\! d-d^{\zeta}).$
\item The complex $(T(V),{^{\epsilon}}\! d)$ is acyclic if $\sigma$ is demi-natural with respect to $w$ and the pair $(w,\epsilon)$ is normalized. 
\item The complex $(T(V),d^{\zeta})$ is acyclic if the pair $(w,\zeta)$ is normalized. 
\end{enumerate}
\end{corollary}

\begin{proof}
 The concatenation map $\ov \mapsto \ov w$ gives the demanded homotopies.
\end{proof}

\subsection{Loday's hyper-boundaries}\label{sec:hyper} 

Our quantum shuffle setting provides a natural interpretation for J.-L.Loday's hyper-boundaries (see \cite{Cyclic}, exercise E.2.2.7), which we redefine as generalizations of the ``shuffle'' differentials from theorem \ref{thm:cuts}. 

\begin{definition}
Let $(V,\sigma)$ be a pre-braided vector space with a braided character $\epsilon.$ The maps
\begin{align*}
V^{\otimes n} & \longrightarrow V^{\otimes (n-k)},\\
{^{\epsilon,(k)}}\! d : \overline{v} & \longmapsto (\epsilon^{\otimes k} \otimes \Id_{n-k}) \circ \Csh^{k,n-k}(\overline{v}), \\
d^{\epsilon,(k)} : \overline{v} & \longmapsto (-1)^{kn-\frac{k(k+1)}{2}} 
(\Id_{n-k}\otimes  \epsilon^{\otimes k} )\circ\Csh^{n-k,k}(\overline{v}) 
\end{align*}
are called \emph{hyper-boundaries} on $T(V).$
\end{definition}

The last sign should be understood as $(-1)^{n-1}(-1)^{n-2}\cdots (-1)^{n-k}.$

For $k=1$ one recovers the braided differentials ${^\epsilon}\! d$ and $d^{\epsilon}.$ 

The next step is to understand compositions of hyper-boundaries, generalizing  
$$d^{(1)}\circ d^{(1)} = 0={^{(1)}}\! d\circ {^{(1)}}\! d.$$ We start with a kind of a special case. This result seems to be well-known, but a written proof is difficult to find in literature. See for example \cite{Lebed}. 
\begin{lemma}
Consider a vector space $W$ and an element $w \in W.$ One has
$$w^{\otimes m} {\underset{-\tau}{\shuffle}} w^{\otimes k} =   \tbinom{m+k}{k}_{-1} w^{\otimes (m+k)},\qquad\qquad \text{} \quad \tbinom{m+k}{k}_{-1} := \begin{cases}
 0 & \text{if }mk\text{ is odd,} \\
 \binom{[(m+k)/2]}{[k/2]} & \text{otherwise.} 
\end{cases} $$
The brackets $[\: \cdot \: ]$ stand here for the lower integral part of a number.
\end{lemma}

This lemma is crucial in the calculations giving
\begin{theorem}\label{thm:hyper}
Let $(V,\sigma)$ be a pre-braided vector space with a braided character $\epsilon.$ One has
$${^{\epsilon,(m)}}\! d \circ {^{\epsilon,(k)}}\! d = \tbinom{m+k}{k}_{-1} {^{\epsilon,(m+k)}}\! d,$$
$$d^{\epsilon,(m)} \circ d^{\epsilon,(k)} = \tbinom{m+k}{k}_{-1} d^{\epsilon,(m+k)}.$$
\end{theorem}

\begin{proof}
We prove the first formula only.
By definition, 
$${^{\epsilon,(m)}}\! d \circ {^{\epsilon,(k)}}\! d (\overline{v}) = (\epsilon \otimes \cdots\otimes\epsilon\otimes \Id_{n-k-m})\circ(\epsilon \otimes \cdots\otimes\epsilon \otimes \Csh^{m,n-k-m})\circ \Csh^{k,n-k}(\overline{v}).$$
By the coassociativity of the quantum co-shuffle comultiplication, it equals
$$(\epsilon \otimes \cdots\otimes\epsilon\otimes \Id_{n-k-m})\circ(\Csh^{k,m} \otimes \Id_{n-k-m})\circ \Csh^{m+k,n-m-k}(\overline{v}).$$
Now $\epsilon$ is a braided character, so 
$$(\epsilon \otimes\epsilon)\circ \sigma = \epsilon \otimes\epsilon = (\epsilon \otimes\epsilon)\circ \tau,$$
thus $${^{\epsilon,(m)}}\! d \circ {^{\epsilon,(k)}}\! d (\overline{v}) = (\epsilon \otimes \cdots\otimes\epsilon\otimes \Id_{n-k-m})\circ({\underset{-\tau}{\cshuffle}}^{k,m} \otimes \Id_{n-k-m})\circ \Csh^{m+k,n-m-k}(\overline{v}).$$
The dual version of the previous lemma calculates
$$(\epsilon \otimes \cdots\otimes\epsilon)\circ {\underset{-\tau}{\cshuffle}}^{k,m} = \tbinom{m+k}{k}_{-1} \epsilon \otimes \cdots\otimes\epsilon,$$
and the previous expression becomes $\tbinom{m+k}{k}_{-1} {^{\epsilon,(m+k)}}\! d.$
\end{proof}

The relations from J.-L.Loday's exercise, which are particular cases of the above theorem for several values of $m$ and $k,$ are thus easily proved and generalized thanks to our quantum co-shuffle interpretation.

\section{Basic examples: familiar complexes recovered}\label{sec:examples}

Now we consider a $\k$-vector space (or an $\kk$-module) $V$ with some {algebraic structure}, and we look for a \textbf{\textit{pre-braiding $\sigma$ encoding the properties of this structure}}. Such pre-braidings are informally called \emph{structural}. Certain algebraic properties of the initial structure are coded by the \textbf{\textit{invertibility}} condition for the corresponding pre-braiding. In each case,  \textbf{\textit{braided characters}} are determined, always \underline{up to scalar multiples}, recovering the usual algebraic notions of characters. Theorem \ref{thm:cuts} then gives numerous bicomplex structures on $T(V).$ We calculate explicitly some of the differentials obtained this way, recognizing many familiar complexes. \textbf{\textit{Arrow operations}} are also considered, showing the triviality of some of the appearing homologies (cf. corollary \ref{crl:hom_operations}). In some cases, $V$ is endowed with a (semi-)pre-braided coalgebra structure, giving, according to theorem \ref{thm:BraidedSimplHom}, a (very) weakly bisimplicial structure on $T(V).$ The \textbf{\textit{comultiplications}} $\Delta$ we use always arise naturally from the original algebraic structure. 

A typical subsection of this section contains thus several lemmas, one for each question emphasized above, followed by propositions explicitly describing the bidifferential or simplicial structures obtained for ``interesting'' characters. Graphical calculus is extensively used.

\subsection{Vector spaces}

Following a nice mathematical tradition, the first example we consider is the trivial one: that of an ``empty'' structure. Take any vector space $V$ and the flip 
$$\tau : v\otimes w\longmapsto w\otimes v$$
 as its braiding. Each $\epsilon \in V^*$ is automatically a braided character. In particular,
$${^\epsilon}\! d = d^\epsilon : v_1 \ldots v_n \longmapsto \sum_{i=1}^n (-1)^{i-1} \epsilon(v_i) v_1 \ldots \widehat{v_i} \ldots v_n $$
gives the well-known \textit{\textbf{Koszul differential}}, in its simplest form.

Further, a ($\tau$-cocommutative) braided coalgebra structure on $(V,\tau)$ is precisely a (cocommutative) comultiplication in the usual sense. The corresponding very weakly simplicial structure on $T(V)$ is simplicial if and only if $\epsilon$ is the \emph{counit} for the comultiplication $\Delta.$ In the last case the cocommutativity is not necessary for the structure to be simplicial. Thus, according to theorem \ref{thm:BraidedSimplHom}, one can quotient the Koszul complex by the images of $s_{n;i}:=\Delta_i$ without changing the homology.

\subsection{Self-distributive structures}\label{sec:racks}

The simplest non-trivial example of a pre-braiding naturally coming from an algebraic structure is the following. Take a set $S$ with a binary operation $ \lhd : S \times S \rightarrow S.$ Define an application 
\begin{align}\label{eqn:RackBraid}
 \sigma = \sigma_\lhd :  S \times S & \longrightarrow S \times S, \notag\\
 (a,b) & \longmapsto (b,a \lhd b).
\end{align} 

It is very familiar to topologists, since it can be interpreted in terms of the fundamental group of the complement of a knot. See for instance the seminal paper \cite{Joyce}, or \cite{Kamada} for a very readable introduction. Graphically $\sigma_\lhd$ becomes
  \begin{center}
\begin{tikzpicture}[scale=0.5]
 \draw (1,0) -- (0,1);
 \draw (0,0) -- (0.4,0.4);
 \draw (0.6,0.6) -- (1,1);
 \node at (1,-0.5)  {$b$};
 \node at (0,-0.5)  {$a$};
 \node at (3,-0.5)  {.};
 \node at (0,1) [above] {$b$}; 
 \node at (1,1) [above right] {$a \lhd b$};
\end{tikzpicture}
   \captionof{figure}{Pre-braiding for shelves}\label{pic:BrShelf}
\end{center}

All the ``braided'' notions are to be understood in the \textbf{{set-theoretic sense}} in this subsection (cf. remark \ref{rmk:SetYB}).

The structure for which $\sigma_\lhd$ is a pre-braiding is well-known (see for instance \cite{Crans} or \cite{Lebed}):

\begin{lemma}\label{thm:shelf} 
The map $\sigma_\lhd$ is a pre-braiding if and only if $\lhd$ is \textbf{\emph{self-distributive}}:
\begin{equation}
  \tag{SD}
  (a\lhd b)\lhd c = (a\lhd c)\lhd(b\lhd c) \qquad \forall a,b,c \in S.
  \label{eqn:SelfDistr'}
\end{equation} 
\end{lemma}

\begin{definition}
A pair $(S,\lhd)$ satisfying \eqref{eqn:SelfDistr'} is called a \emph{shelf}, or a \emph{self-distributive system}. 
\end{definition}

\begin{lemma}
A map $f$ between two shelves $(S_1,\lhd_1)$ and $(S_2,\lhd_2)$ is a shelf morphism (i.e. $f(a \lhd_1 b) = f(a) \lhd_2 f(b)$ ) if and only if it is a braided morphism from $(S_1,\sigma_{\lhd_1})$ to $(S_2,\sigma_{\lhd_2}).$
\end{lemma}
These ``if and only if'' lemmas show that the pre-braiding $\sigma_\lhd$ \textbf{{encodes}} the defining property of a shelf, just as we wanted.

Fix a \underline{shelf} $(S,\lhd)$ until the end of this section.

\begin{lemma}
The pre-braiding $\sigma_\lhd$ is a braiding if and only if the application $a\mapsto a\lhd b$ is a bijection on $S$ for every $b \in S,$ that is if there exists an application $\widetilde{\lhd}: S \times S\longrightarrow S$ such that
\begin{equation}\label{eqn:Rack'}\tag{R}
  (a\lhd b)\widetilde{\lhd} b = (a\widetilde{\lhd} b)\lhd b = a \qquad \forall a,b \in S.
\end{equation} 
\end{lemma}

\begin{definition}
A triple $(S,\lhd, \widetilde{\lhd})$ satisfying \eqref{eqn:SelfDistr'} and \eqref{eqn:Rack'} is called \emph{a rack}.
\end{definition}

Now \textit{\textbf{linearize}} a shelf $(S,\lhd)$: put $V:=\k S,$ where $\k$ is a field, and extend the braiding $\sigma_\lhd$ to $V$ linearly. One gets a pre-braided vector space $(V,\sigma_\lhd).$ 

\begin{lemma}\label{thm:ShelfCoChar}
Characters $\epsilon \in V^*$ are characterized by $\epsilon (a\lhd b) = \epsilon (a)$ forall $a,b \in S$ with $\epsilon(b) \neq 0.$
\end{lemma}

In the $\kk$-linear setting, this condition is sufficient but not necessary in general.

Here are some examples of braided characters:
\begin{example}
\begin{enumerate}
\item The linearization of
\begin{equation}\label{eqn:ShelfCounit}
\varepsilon : a \mapsto 1\; \forall a \in S 
\end{equation}
 is always a character.
\item The linearization of a \emph{\textbf{``Dirac map''}}
\begin{equation}\label{eqn:Dirac}
\varphi_a(b) := \delta_{a,b}= \begin{cases}
 1 & \text{if } b=a, \\
 0 & \text{for other } b \in S
\end{cases}
\end{equation}
(here $\delta_{a,b}$ is the Kronecker delta) is a character if and only if $a$ is idempotent and satisfies $b \lhd a \neq a$ for $b\neq a.$ In particular, if $S$ is a \emph{\textbf{quandle}}, i.e. a rack with idempotent elements:
\begin{equation}\label{eqn:Quandle'}\tag{Q}
  a\lhd a = a   \qquad \forall a \in S,
\end{equation} 
 then all the $\varphi_a$'s are characters. 
\end{enumerate}
\end{example}

We finish with a more conceptual construction of a class of braided characters. Recall that a character for an algebraic structure is usually defined as a morphism to the trivial structure. It is natural (having in mind the conjugation quandle) to define the \emph{\textbf{trivial shelf structure}} on a set $X$ by $x \lhd y = x$ for all $x,y\in X.$

\begin{definition}
A \emph{shelf character} for a shelf $(S,\lhd)$ is a shelf morphisms $\epsilon:S \rightarrow \k,$ where $\k$ is endowed with the trivial shelf structure. In other words, $\epsilon$ satisfies
$$ \epsilon (a\lhd b) = \epsilon (a) \qquad \forall a,b \in S. $$
\end{definition}  

Lemma \ref{thm:ShelfCoChar} then implies
\begin{lemma}
 The linearization of a shelf character is always a braided character for the pre-braiding $\sigma_\lhd.$
 Moreover, two braided characters coming from shelf characters are automatically $\sigma_\lhd$-compatible. 
\end{lemma}

\medskip
The last ingredient we need is a comultiplication. The one proposed here is quite classical in the self-distributive world:
\begin{lemma}
Let $\Delta_D:V \rightarrow V \otimes V$ be the linearization of the \emph{{diagonal map}}
\begin{align*}
D:S &\longrightarrow S \times S,\\
a &\longmapsto (a,a) \qquad \forall a \in S.
\end{align*}
\begin{enumerate}
\item $(V,\sigma_\lhd,\Delta_D)$ is a semi-pre-braided coalgebra.
\item This coalgebra is pre-braided if and only if
$$a \lhd b = (a \lhd b ) \lhd b \qquad \forall a,b \in S.$$
\item The $\sigma_\lhd$-cocommutativity for  $\Delta_D$ is equivalent to \eqref{eqn:Quandle'}. 
\end{enumerate}
\end{lemma}

\begin{definition}
A shelf satisfying \eqref{eqn:Quandle'} is called a \emph{{spindle}}.
\end{definition}

\begin{remark}\label{rmk:ShelfDegenerateCx}
The image of the map $s_{n;i}:=(\Delta_D)_i$ is the linear span of the elements $(a_1,\ldots, a_{n+1}) \in S^{\times (n+1)}$ with $a_i = a_{i+1}.$
\end{remark} 

\medskip
It is now time to put together all the ingredients and to make some concrete calculations of bidifferentials. Start with the character $\varepsilon$ defined by \eqref{eqn:ShelfCounit}.

\begin{proposition}\label{thm:SRQHom}
Take a shelf $(S, \lhd).$
\begin{enumerate}
\item 
A bicomplex structure can be defined on $T(\k S)$ by
\begin{align*}
{^{\varepsilon}}\! d(a_1,\ldots, a_n) &= \sum_{i=1}^n (-1)^{i-1}((a_1\lhd a_i),\ldots,(a_{i-1}\lhd a_i),a_{i+1},\ldots, a_n),\\
d^{\varepsilon}(a_1,\ldots, a_n) &= \sum_{i=1}^n (-1)^{i-1} (a_1, \ldots\widehat{a_i}\ldots, a_n). 
\end{align*}

\item This bidifferential comes from a pre-bisimplicial structure given by
\begin{align*}
d_{n;i}(a_1,\ldots, a_n) &= ((a_1\lhd a_i),\ldots,(a_{i-1}\lhd a_i),a_{i+1},\ldots, a_n),\\
d'_{n;i}(a_1,\ldots, a_n) &= (a_1, \ldots\widehat{a_i}\ldots, a_n). 
\end{align*}

\item\label{it:SpindleSimpl} If our shelf is moreover a spindle, then $(T(\k S),d_{n;i},d'_{n;i}, s_{n;i}:=(\Delta_D)_i)$ is a weakly bisimplicial vector space. As a consequence, the linear span $C^D_*(S)$ of the elements $(a_1,\ldots, a_n)\in S^{\times n}$  with $a_i = a_{i+1}$ for one of the subscripts $i$ forms a sub-bicomplex of $(T(\k S), {^{\varepsilon}}\! d,d^{\varepsilon}).$ 
\end{enumerate}
\end{proposition}

\begin{proof}
Points 1 and 2 are direct applications of theorems \ref{thm:cuts} and \ref{thm:BraidedSimplHom} respectively to the pre-braiding $\sigma_\lhd$ from \eqref{eqn:RackBraid} and the character $\varepsilon,$  combined with the lemmas from this section. 

As for point \ref*{it:SpindleSimpl}, theorem \ref{thm:BraidedSimplHom} gives only a half of this assertion: $(T(\k S),d_{n;i},s_{n;i})$ is a weakly simplicial vector space, hence $C^D_*(S)$ is a subcomplex of $(T(\k S), \partial = {^{\varepsilon}}\! d).$ Since $(\k S,\sigma_\lhd,\Delta_D)$ is only a \underline{semi}-pre-braided coalgebra in general, the compatibilities \eqref{eqn:simpl3} between the $d'_{n;i}$'s and the $s_{n;i}$'s should be verified by hand, which is an easy exercise. Finally, the explicit description of the degenerate sub-bicomplex follows from remark \ref{rmk:ShelfDegenerateCx}.
\end{proof} 

Let us point out familiar complexes recovered in this proposition. We freely replace the field $\k$ by $\ZZ$ since, as noted above, all the constructions here work in the $\kk$-linear setting.

\begin{example}
\begin{enumerate}
\item The \textit{\textbf{rack homology}} (\cite{RackHom}) appears as 
$$C^R_*(S):=(T(\ZZ S),{^{\varepsilon}}\! d - d^{\varepsilon}).$$
\item The \textit{\textbf{shelf, or one-term distributive, homology}} (\cite{Prz1}, \cite{PrzSikora}) appears as 
$$C^{\lhd}_*(S):=(T(\ZZ S),{^{\varepsilon}}\! d).$$
\item  The \textit{\textbf{quandle homology}} (\cite{QuandleHom}) is the quotient 
$$C^Q_*(S) := C^R_*(S) / C^D_*(S).$$
\end{enumerate}
\end{example}

\medskip
The map ${^\epsilon}\!\pi$ takes the familiar {\textit{diagonal}} form here:
$${^\epsilon}\!\pi_b(a_1,\ldots, a_n)= (a_1 \lhd b, \ldots, a_n\lhd b).$$
Moreover, $\varepsilon$ is right $\sigma_\lhd$-compatible and forms a normalized pair with any $b \in S.$ Proposition \ref{thm:hom_operations} and corollary \ref{crl:hom_operations} are then applicable, recovering some results on homology operations from \cite{Prz2}, \cite{Prz1} and \cite{PrzSikora} and their consequences:

\begin{proposition}
\begin{enumerate}
\item The complex  $(T(\k S),d^{\varepsilon})$ is acyclic.
\item If there exists an element $b \in S$ such that the application $a\mapsto a\lhd b$ is a bijection on $S,$ then the complex  $(T(\k S), {^{\varepsilon}}\! d)$ is acyclic. 
\item If there exists an $a \in S$ stable by all the inner shelf morphisms, i.e.
\begin{equation}\label{eqn:FixedPoint}
a\lhd b = a \qquad \forall b \in S,
\end{equation}
then the complex  $(T(\k S), {^{\varepsilon}}\! d)$ is acyclic. 
\end{enumerate}
\end{proposition}

\begin{proof}
Point 1 follows from corollary \ref{crl:hom_operations} (point 3).

Point 2 follows from corollary \ref{crl:hom_operations} (point 1), since the arrow  operation ${^\epsilon}\!\pi_b,$ shown there to be homotopic to zero, is now invertible, the inverse given by
$(a_1,\ldots, a_n)\longmapsto (a_1 \wlhd b, \ldots, a_n\wlhd b),$ where $a\mapsto a\wlhd b$ denotes the map inverse to $a\mapsto a\lhd b.$

Point 3 follows from the ``right'' version of corollary \ref{crl:hom_operations} (point 3): condition \eqref{eqn:FixedPoint} means precisely that $a$ is left $\sigma_\lhd$-compatible with $\varepsilon.$ 
\end{proof}

Thus, the complex $(T(\k S), {^{\varepsilon}}\! d)$ is acyclic for a rack. However, it can be highly non-trivial for shelves (cf. (\cite{Prz1}, \cite{PrzSikora}).

\medskip

Further, let us turn to the characters given by Dirac maps \eqref{eqn:Dirac}. Theorem \ref{thm:BraidedSimplHom} applied to the pre-braiding $\sigma_\lhd$ and the character $\varphi_a$ gives

\begin{proposition}
Take a quandle $(S, \lhd, \wlhd)$ with a fixed element $a.$ A pre-bisimplicial structure can be given on $T(\k S)$ by
\begin{align*}
d_{n;i}(a_1,\ldots, a_n) &= \delta_{a,a_i}((a_1\lhd a_i),\ldots,(a_{i-1}\lhd a_i),a_{i+1},\ldots, a_n),\\
d'_{n;i}(a_1,\ldots, a_n) &= \delta_{a,(\cdots(a_i\lhd a_{i+1})\lhd \cdots )\lhd a_n}  (a_1, \ldots\widehat{a_i}\ldots, a_n). 
\end{align*}

Moreover, the face maps $d_{n;i}$ combined with degeneracies $s_{n;i}:=\Delta_i$ give a weakly simplicial structure.
\end{proposition}

The total differentials 
$${^{\varphi_a}}\! d(a_1,\ldots, a_n) = \sum_{i=1}^n (-1)^{i-1}\delta_{a,a_i}((a_1\lhd a_i),\ldots,(a_{i-1}\lhd a_i),a_{i+1},\ldots, a_n)$$
are called \textbf{\emph{partial derivatives}} and are denoted by $\dfrac{\partial^1}{\partial a}$ in \cite{Prz2}. Our general setting thus contains some results of \cite{Prz2}.

\begin{remark}
One can not talk about weakly \underline{bi}simplicial structure here, since the coalgebra $(\k S,\sigma_\lhd,\Delta_D)$ is only \underline{semi}-braided, and the compatibilities \eqref{eqn:simpl3} between the $d'_{n;i}$'s and the $s_{n;i}$'s, which are automatical for pre-braided coalgebras and happen to hold in proposition \ref{thm:SRQHom}, are no longer true for here. However, $C^D_*(S)$ is still a sub-bicomplex of $(T(\kk S),{^{\varphi_a}}\! d,d^{\varphi_a})$: indeed,
$d'_{n+1;i} \circ s_{n;j} (a_1,\ldots,a_n)$ is proportional to $s_{n-1;j-1}(a_1,\ldots,\widehat{a_i}, \ldots,a_n)$ and is thus still in the image of $s_{n-1;j-1} $ for all $1\le i < j \le n.$
\end{remark}

\medskip
We finish with an example where different characters are used on the right and on the left. This construction is inspired by \cite{TwistedQuandle}.

\begin{proposition}
 Take a shelf $(S, \lhd)$ and work with its linearization $\Lambda S,$ $\Lambda = \ZZ[T^{\pm 1}].$ The pre-braiding $\sigma_\lhd,$  combined with the shelf characters $\varepsilon$ and 
\begin{align*}
\epsilon_T : \Lambda S &\longrightarrow \Lambda,\\
a &\longmapsto T \quad \forall a \in S
\end{align*}
 define, via theorem \ref{thm:cuts}, a bicomplex structure on $T(\Lambda S)$ by
\begin{align*}
{^{\varepsilon}}\! d(a_1,\ldots, a_n) &= \sum_{i=1}^n (-1)^{i-1}((a_1\lhd a_i),\ldots,(a_{i-1}\lhd a_i), a_{i+1}, \ldots, a_n),\\
d^{\epsilon_T}(a_1,\ldots, a_n) &=\sum_{i=1}^n (-1)^{i-1} T (a_1, \ldots\widehat{a_i}\ldots, a_n),
\end{align*}
refined, as usual, to a pre-bisimplicial structure.
\end{proposition}

The differential ${^{\varepsilon}}\! d - d^{\epsilon_T}$ defines the \textbf{\textit{twisted rack homology}} from \cite{TwistedQuandle}.

\subsection{Associative algebras}\label{sec:bar}

Now take a $\k$-vector space $V$ endowed with a bilinear operation $\mu : V\otimes V \rightarrow V$ and a distinguished element $\one \in V,$ sometimes regarded as a linear map 
\begin{align*}
\nu:\k &\longrightarrow V, \\
 \alpha &\longmapsto \alpha\one.
\end{align*} 
In this subsection we construct quite an exotic non-invertible pre-braiding on $V$ which encodes the associativity of $\mu,$ and complete it with an exotic comultiplication.
 
 \medskip

Consider the bilinear application 
\begin{align}\label{eqn:AssBraid}
 \sigma = \sigma_\mu : V \otimes V & \longrightarrow V \otimes V, \notag\\
 v\otimes w & \longmapsto \one \otimes \mu(v\otimes w)
\end{align} 
or, graphically,
  \begin{center}
\begin{tikzpicture}[scale=0.5]
 \draw (0,0) -- (1,1);
 \draw (1,0) -- (0.5,0.5);
 \draw (0.3,0.7) -- (0,1);
 \node at (0.5,0.5) [right] {$\mu$};
 \node at (0.3,0.7) [left] {$\nu$}; 
 \node at (1,0) [below] {$w$};
 \node at (0,0) [below] {$v$};
 \node at (0,1) [above] {$\one$}; 
 \node at (1,1) [above right] {$\mu(v\otimes w)$};
 \fill[orange] (0.3,0.7) circle (0.1);
 \fill[teal] (0.5,0.5) circle (0.1);
 \node at (2,0) {.};
\end{tikzpicture}
   \captionof{figure}{Pre-braiding for associative algebras }\label{pic:BrAss}
 \end{center}

\begin{lemma} \label{thm:Bar} 
Suppose that $\one$ is a right unit for $\mu,$ i.e. 
$\mu(v \otimes \one) = v \; \forall v \in V.$
 Then the map $\sigma_\mu$ is a pre-braiding if and only if $\mu$ is associative on $V,$ i.e.
 \begin{equation}\label{eqn:ass}\tag{Ass}
\mu(\mu(v\otimes w)\otimes u)=\mu(v\otimes \mu(w\otimes u)) \qquad\forall v,w,u \in V,
\end{equation}
\begin{center}
\begin{tikzpicture}[scale=0.2]
\draw (0,0)--(2,2);
\draw (4,0)--(2,2);
\draw (1,1)--(2,0);
\draw (2,3)--(2,2);
\fill[teal] (1,1) circle (0.2);
\fill[teal] (2,2) circle (0.2);
\node at (6,1.5) {$=$};
\node at (7,1.5) {};
\end{tikzpicture}
\begin{tikzpicture}[scale=0.2]
\draw (0,0)--(2,2);
\draw (4,0)--(2,2);
\draw (3,1)--(2,0);
\draw (2,3)--(2,2);
\fill[teal] (3,1) circle (0.2);
\fill[teal] (2,2) circle (0.2);
\node at (5,0.5) {.};
\end{tikzpicture}
   \captionof{figure}{Associativity }\label{pic:Ass}
\end{center}
\end{lemma}

\begin{proof}
Graphically, YBE for $\sigma_\mu$ means 

\begin{center}
\begin{tikzpicture}[scale=0.8]
\draw [rounded corners](0,0)--(0,0.25)--(1,0.75)--(1,1.25)--(2,1.75)--(2,3);
\draw [rounded corners](1,0)--(1,0.25)--(0.5,0.5);
\draw [rounded corners](0.3,0.65)--(0,0.75)--(0,2.25)--(1,2.75)--(1,3);
\draw [rounded corners](2,0)--(2,1.25)--(1.5,1.5);
\draw [rounded corners](1.3,1.65)--(1,1.75)--(1,2.25)--(0.5,2.5);
\draw [rounded corners](0.3,2.65)--(0,2.75)--(0,3);
\node [below] at (0,0) {$v$};
\node [below] at (1,0) {$w$};
\node [below] at (2,0) {$u$};
\node [above] at (0,3) {$\one$};
\node [above] at (1,3) {$\one$};
\node [above right,blue] at (1.5,3) {\underline{$\mu(\mu(v\otimes w)\otimes u)$}};
\node [left] at (0,1) {$\one$};
\node at (1,1) {$\mu(v\otimes w)$};
\node [right] at (1,2) {$\one$};
\fill[orange] (0.3,0.65) circle (0.1);
\fill[orange] (0.3,2.65) circle (0.1);
\fill[orange] (1.3,1.65) circle (0.1);
\fill[teal] (0.5,0.5) circle (0.1);
\fill[teal] (0.5,2.5) circle (0.1);
\fill[teal] (1.5,1.5) circle (0.1);
\node  at (4,1.5){$=$};
\end{tikzpicture}
\begin{tikzpicture}[scale=0.8]
\node  at (-2,1.5){};
\draw [rounded corners](1,1)--(1,1.25)--(2,1.75)--(2,3.25)--(1.5,3.5);
\draw [rounded corners](1.3,3.65)--(1,3.75)--(1,4);
\draw [rounded corners](0,1)--(0,2.25)--(1,2.75)--(1,3.25)--(2,3.75)--(2,4);
\draw [rounded corners](2,1)--(2,1.25)--(1.5,1.5);
\draw [rounded corners](1.3,1.65)--(1,1.75)--(1,2.25)--(0.5,2.5);
\draw [rounded corners](0.3,2.65)--(0,2.75)--(0,4);
\node [below] at (0,1) {$v$};
\node [below] at (1,1) {$w$};
\node [below] at (2,1) {$u$};
\node at (4,1) {.};
\node [above] at (0,4) {$\one$};
\node [above] at (1,4) {$\one$};
\node [above right,blue] at (1.5,4) {\underline{$\mu(v\otimes \mu(w\otimes u))$}};
\node [left] at (0,3) {$\one$};
\node [left] at (1,3) {$v$};
\node [right] at (1,2) {$\one$};
\node [right] at (2,2) {$\mu(w\otimes u)$};
\fill[orange] (1.3,1.65) circle (0.1);
\fill[orange] (0.3,2.65) circle (0.1);
\fill[orange] (1.3,3.65) circle (0.1);
\fill[teal] (1.5,3.5) circle (0.1);
\fill[teal] (0.5,2.5) circle (0.1);
\fill[teal] (1.5,1.5) circle (0.1);
\end{tikzpicture}
   \captionof{figure}{Pictorial proof of lemma \ref{thm:Bar} }\label{pic:ProofBar}
\end{center}
This is equivalent to the associativity condition \eqref{eqn:ass} for $\mu.$
\end{proof}

One thus gets, like in the case of shelves, a pre-braiding subtly encoding the algebraic structure ``associative algebra''.

\begin{remark}
The braiding $\sigma_\mu$ is \textbf{highly non-invertible}. More precisely, $\sigma_\mu^2 = \sigma_\mu$ if $\one$ is moreover a left unit.
\end{remark}

\medskip

Fix an \underline{associative $\k$-algebra} $(V, \mu)$ with a \underline{right unit} $\one$ until the end of this section. Such algebras are called \textbf{\textit{right-unital}} here.

\begin{definition}
An \emph{algebra character} is a unital algebra morphism $\epsilon: V \rightarrow \k$ to the trivial algebra $\k,$ i.e.
\begin{align}\label{eqn:char}
\epsilon(\mu(v \otimes w))&=\epsilon(v)\epsilon(w) \qquad \forall v,w \in V, \\
\epsilon (\one) &=1, \notag
\end{align}
\begin{center}
\begin{tikzpicture}[scale=0.5]
 \draw (0,0) -- (0.5,0.5); 
 \draw  (0.5,0.5) -- (0.5,1);
 \draw (1,0) -- (0.5,0.5);
 \node at (0.5,0.5) [right]{$\mu$};
 \node at (1,0) [below] {$w$};
 \node at (0,0) [below] {$v$};
 \fill [teal] (0.5,0.5) circle (0.1);
 \fill [violet] (0.5,1) circle (0.1);
 \node at (0.5,1) [above,violet] {$\epsilon$};
 \node at (2,0.5) {$=$};
 \draw  (3,0) -- (3,1);
 \draw  (4,0) -- (4,1);
 \fill [violet] (3,1) circle (0.1);
 \fill [violet] (4,1) circle (0.1);
 \node at (3,1) [above,violet] {$\epsilon$};
 \node at (4,1) [above,violet] {$\epsilon$};
 \node at (4,0) [below] {$w$};
 \node at (3,0) [below] {$v$};
 \node at (5,0) {,};
 \node at (7,0) {};
\end{tikzpicture}
\begin{tikzpicture}[scale=0.5]
 \draw (0,0) -- (0,1); 
 \fill [orange] (0,0) circle (0.1);
 \fill [violet] (0,1) circle (0.1);
 \node at (0,0) [below] {$\nu$};
 \node at (0,1) [above,violet] {$\epsilon$};
 \node at (1,0.5) {$= 1.$};
\end{tikzpicture}
   \captionof{figure}{Algebra character }\label{pic:CharAss}
\end{center}
A \emph{non-unital algebra character} satisfies the first condition only.
\end{definition}

\begin{lemma}
Braided characters are precisely maps $\epsilon \in V^*$ satisfying
\begin{equation}\label{eqn:BrCharUAA}
\epsilon (\one) \epsilon(\mu(v \otimes w))=\epsilon(v)\epsilon(w) \quad \forall v,w \in V.
\end{equation}
In particular, every algebra character is a braided character. On the other hand, any non-zero solution of \eqref{eqn:BrCharUAA} is a scalar multiple of an algebra character.
\end{lemma}

Working over a ring $\kk$ instead of a field $\k,$ one has to drop the last statement. 

\begin{lemma}
The linear map 
\begin{align*}
\Delta_{\one}: V &\longrightarrow V \otimes V,\\
v &\longmapsto \one \otimes v
\end{align*}
endows $(V, \sigma_{\mu})$ with a pre-braided coalgebra structure, $\sigma_{\mu}$-cocommutative if and only if $\one$ is also a left unit.
\end{lemma}

The last remarks concern arrow operations and the special role of the right unit $\one$ in our braided story. Recall definition \ref{def:NormCompat} and corollary \ref{crl:hom_operations}.
 
\begin{lemma}\label{thm:OneAss}
\begin{enumerate}
\item The arrow operations give \textit{{peripheral}} actions:
$${^\epsilon}\!\pi_w(v_1 \ldots v_{n-1} v_n) =\epsilon(\one) v_1 \ldots v_{n-1} \mu(v_n \otimes w) \qquad \forall v_i, w \in V.$$
\item In particular, the right unit $\one$ acts by identity if $\varepsilon$ is an algebra character: 
$${^\epsilon}\!\pi_\one = \Id_{T(V)}.$$
\item The pre-braiding $\sigma_\mu$ is demi-natural with respect to $\one.$ Consequently, $\one$ is right $\sigma_\mu$-compatible with any $f \in V^*.$ 
\end{enumerate}
\end{lemma}

\medskip
We now turn to concrete computations: 

\begin{proposition}\label{thm:ComputAss}
Take a right-unital associative algebra $(V,\mu,\one)$ and two algebra characters $\epsilon$ and $\zeta.$ 
\begin{enumerate}
\item 
A bicomplex structure can be defined on $T(V)$ by
\begin{align*}
{^{\epsilon}}\! d(v_1\ldots v_n) &=\epsilon (v_1) v_2\ldots v_n +\sum_{i=1}^{n-1} (-1)^{i}v_1\ldots v_{i-1}\mu(v_i\otimes v_{i+1})v_{i+2}\ldots v_n,\\
d^{\zeta}(v_1\ldots v_n) &=(-1)^{n-1}\zeta (v_n) v_1\ldots v_{n-1} \\
&+\sum_{i=0}^{n-2} (-1)^{i} \zeta(v_{i+1})\cdots \zeta(v_n)v_1\ldots v_i\one \one\ldots \one.
\end{align*}
\item\label{it:AssPreSimpl} This bidifferential comes from the pre-bisimplicial structure
\begin{align*}
d_{n;1}(v_1\ldots v_n) &= \epsilon (v_1) v_2\ldots v_n,\\
d_{n;i+1}(v_1\ldots v_n) &= v_1\ldots v_{i-1}\mu(v_i\otimes v_{i+1})v_{i+2}\ldots v_n, \qquad 1 \le i \le n-1,\\
d'_{n;i}(v_1\ldots v_n) &= \zeta(v_{i})\cdots \zeta(v_n)v_1\ldots v_{i-1}\one \one\ldots \one, \qquad 1 \le i \le n-1,\\
d'_{n;n}(v_1\ldots v_n) &= \zeta (v_n) v_1\ldots v_{n-1}.
\end{align*}

\item\label{it:BarAcyclic}
The complex $(T(V),{^{\epsilon}}\! d)$ is acyclic.

\item\label{it:AssWSimpl}
If $\one$ is a two-sided unit, then the above structure becomes weakly bisimplicial, with
$$s_{n;i}(v_1\ldots v_n) = v_1\ldots v_{i-1}\one v_i\ldots v_n, \qquad 1 \le i \le n.$$

\item\label{it:AssSimpl}
In this case the structure $(T(V),d_{n;i},s_{n;i})$ is even simplicial.

\item
In the normalized bicomplex, $d'_{n;i}=0$ for $i < n-1.$

\item\label{it:GroupHom}
Still supposing the unit $\one$ two-sided, the differential ${^{\epsilon}}\! d - d^{\zeta}$ descends to $T(V'),$ where $V':=V/\k \one,$ giving the differential
\begin{align*}
{^{\epsilon}}\! d^{\zeta}(v_1\ldots v_n) &:=\epsilon (v_1) v_2\ldots v_n \\
&+\sum_{i=1}^{n-1} (-1)^{i}v_1\ldots v_{i-1}\mu(v_i\otimes v_{i+1})v_{i+2}\ldots v_n,\\
&+(-1)^{n}\zeta (v_n) v_1\ldots v_{n-1}.
\end{align*}
\end{enumerate}
\end{proposition}

\begin{proof}
Most of the assertions follow from theorems \ref{thm:cuts} and \ref{thm:BraidedSimplHom} applied to the pre-braiding $\sigma_\mu$ from \eqref{eqn:AssBraid} and the algebra, hence braided, characters $\epsilon$ and $\zeta,$ combined with preceding lemmas.

Point \ref*{it:BarAcyclic} is the corollary \ref{crl:hom_operations} applied to the element $\one,$ posseding the ``nice'' properties described in lemma \ref{thm:OneAss}. Point \ref*{it:AssSimpl} also follows from the properties of $\one.$

More work is needed for proving point \ref*{it:GroupHom}. Point \ref*{it:AssWSimpl} ensures that tensors $v_1\ldots v_{i-1}\one v_i\ldots v_n$ with $1 \le i \le n$ generate a sub-bicomplex of $T(V),$ hence a subcomplex of $(T(V),{^{\epsilon}}\! d - d^{\zeta}).$ Further, proposition \ref{thm:hom_operations} implies that the concatenation map $\ov \mapsto \ov \one$ is a differential complex endomorphism of $(T(V),{^{\epsilon}}\! d - d^{\zeta}),$ thus its image $T(V) \otimes \one$ is a subcomplex. Forming the quotient by these two subcomplexes, one gets the desired differential on $T(V').$
\end{proof}

Differential ${^{\epsilon}}\! d^{\zeta}$ defines a (generalization of a) homology sometimes called the \textbf{\textit{group homology}}, which can also be regarded as the \textit{Hochschild homology with trivial coefficients}.

We finish with a ``non-unital'' remark:
\begin{remark}\label{rmk:nonunit}
In the \textit{non-unital case}, i.e. when $V$ is endowed with a bilinear operation $\mu$ only, one enriches $V$ with a formal two-sided unit $\wV := V \oplus \k \one,$ extending $\mu$ by
$$\mu (\one \otimes v) = \mu (v \otimes \one) = v \qquad \forall v \in \wV.$$
Due to the equivalence of the associativity of $\mu$ on $V$ and on $\wV,$ lemma \ref{thm:Bar} asserts that $\sigma_\mu$ is a pre-braiding \underline{on $\wV$} if and only if $\mu$ is associative \underline{on $V$}. Take the character $\varepsilon (V)\equiv 0, \varepsilon(\one)=1$ on $\wV.$ The differential ${^{\varepsilon}}\! d^{\varepsilon}$ descends to $T((\wV)')\simeq T(V),$ as explained in the previous proposition. One recovers the well-known \textbf{\textit{bar (or standard) differential}}:
$$ d_{bar}(v_1\ldots v_n) = \sum_{i=1}^{n-1} (-1)^{i}v_1\ldots v_{i-1}\mu(v_i\otimes v_{i+1})v_{i+2}\ldots v_n.$$
Moreover, a non-unital algebra character $\epsilon \in V^*$ extends to an algebra character on $\wV$ by imposing $\epsilon(\one)=1.$ Two such non-unital algebra characters then define a differential ${^{\epsilon}}\! d^{\zeta}$ on $T((\wV)')\simeq T(V).$ 
\end{remark}

This trick of adding formal elements will often be handy in what follows.

\begin{remark}
One can also obtain the bar differential without doing this formal unit gymnastics. It suffices to replace the total differential with a ``cut version'' $\partial_{cut}:=\sum_{i=2}^n d_{n;i}$ for the pre-simplicial structure from point \ref{it:AssPreSimpl} of proposition \ref{thm:ComputAss}.
\end{remark}

\subsection{Leibniz algebras}\label{sec:Lei}

 Leibniz algebras are ``non-commutative'' versions of Lie algebras. They were discovered by A.Bloh in 1965, but it was J.-L.Loday who woke the general interest in this structure around 1989 by, firstly, lifting the classical Chevalley-Eilenberg boundary map from the exterior to the tensor algebra, which yields a new interesting chain complex, and, secondly, by observing that the antisymmetry condition could be omitted (cf. \cite{Cyclic},\cite{LoLei},\cite{LoPi},\cite{Cuvier}). Here we recover Loday's complex guided by our ``braided'' considerations. Our interpretation explains the somewhat mysterious element ordering and signs in the formula given by Loday. 
 
 We give here short statements of the main results only, since the proofs are analogous to the associative algebra case; see \cite{Lebed} for details.

\begin{lemma}\label{thm:Lei} 
Take a $\k$-vector space $V$ equipped with a bilinear operation $[,] : V\otimes V \rightarrow V$ and a \textbf{\textit{Lie unit}}, i.e. a central element, $\one$ in $V$: 
\begin{equation}\label{eqn:LieUnit}
[\one , v] = [v , \one] = 0 \qquad \forall v \in V.
\end{equation} 

Then the bilinear application 
\begin{align}\label{eqn:LeiBraid}
 \sigma = \sigma_{[,]} : V \otimes V & \longrightarrow V \otimes V, \notag \\ 
 v\otimes w & \longmapsto w \otimes v + \one \otimes [v, w]
\end{align} 

 is a pre-braiding if and only if 
\begin{equation}\label{eqn:Lei}  \tag{Lei}
  [v,[w,u]]=[[v,w],u]-[[v,u],w] \qquad \forall v,w,u \in V, 
\end{equation} 
\begin{center}
\begin{tikzpicture}[scale=0.2]
\draw (0,0)--(2,2);
\draw (4,0)--(2,2);
\draw (3,1)--(2,0);
\draw (2,3)--(2,2);
\fill[teal] (3,1) circle (0.2);
\fill[teal] (2,2) circle (0.2);
\node at (5,1.5) {$=$};
\node at (6,1.5) {};
\end{tikzpicture}
\begin{tikzpicture}[scale=0.2]
\draw (0,0)--(2,2);
\draw (4,0)--(2,2);
\draw (1,1)--(2,0);
\draw (2,3)--(2,2);
\fill[teal] (1,1) circle (0.2);
\fill[teal] (2,2) circle (0.2);
\node at (5,1.5) {$-$};
\node at (6,1.5) {};
\end{tikzpicture}
\begin{tikzpicture}[scale=0.2]
\draw (0,0)--(2,2);
\draw (4,0)--(1,1);
\draw (2,2)--(2,0);
\draw (2,3)--(2,2);
\fill[teal] (1,1) circle (0.2);
\fill[teal] (2,2) circle (0.2);
\node at (6,0.5) {.};
\end{tikzpicture}
   \captionof{figure}{Leibniz condition }\label{pic:Lei}
\end{center}
\end{lemma}

Note that for the ``only if'' part of the statement, it is essential to work over a field $\k,$ or to demand another technical condition (cf. lemma \ref{thm:encode_cat}). 

\begin{definition}
A pair $(V,[,])$ satisfying \eqref{eqn:Lei} is a \emph{Leibniz algebra}, called \emph{unital} if endowed with a Lie unit $\one.$ 
\end{definition}

One gets the notion of \emph{Lie algebra} when adding the antisymmetry condition.

\begin{lemma}
The pre-braiding $\sigma_{[,]}$ is invertible, the inverse given by
$$\sigma_{[,]}^{-1}: v\otimes w \longmapsto w \otimes v -[w,v] \otimes\one.$$
\end{lemma}

\begin{definition}
A \emph{Lie (or Leibniz) character} is a unital Leibniz algebra morphism $\epsilon: V \rightarrow \k$ to the trivial (with the zero bracket) unital (with $1 \in \k$ as a Lie unit) Lie algebra $\k,$ i.e.
\begin{align}
\epsilon([v,w])&=0 \qquad \forall v,w \in V,\label{eqn:LeiChar}\\
\epsilon(\one)&=1.\notag 
\end{align}
A \emph{non-unital Lie character} satisfies the first condition only.
\end{definition}

\begin{lemma}
A Lie character is automatically a braided character for the braiding $\sigma_{[,]}.$
\end{lemma}

The comultiplication we choose for Leibniz algebras is what one expects: 

\begin{definition}
We call a unital Leibniz algebra $V$ \emph{split} if there is a Leibniz sub-algebra $V'$ of $V$ and a Leibniz algebra decomposition
$ V \simeq V' \oplus \k \one.$
\end{definition}

\begin{lemma}
Take  a split unital Leibniz algebra $(V, \sigma_{[,]},\one).$ The linear map 
\begin{align*}
\Delta_{pr}: V &\longrightarrow V \otimes V,\\
v &\longmapsto \one \otimes v + v \otimes \one \qquad \forall v \in V',\\
\one &\longmapsto \one \otimes \one
\end{align*}
endows it with a semi-braided $\sigma_{[,]}$-cocommutative coalgebra structure.
\end{lemma}

This comultiplication turns all the elements of $V'$ into primitive ones.

Like in the associative algebra case, the Lie unit $\one$ enjoys important properties with respect to arrow operations:

\begin{lemma}\label{thm:OneLei}
\begin{enumerate}
\item The arrow operations give \textit{{adjoint actions}}: 
$${^\epsilon}\!\pi_w(v_1 \ldots v_n) =\epsilon(\one)\sum_{i=1}^n v_1 \ldots [v_i,w] \ldots v_n + \epsilon(w) v_1 \ldots v_n \qquad \forall v_i,w \in V.$$
\item In particular, the Lie unit $\one$ acts by scalars: 
$${^\epsilon}\!\pi_\one = \varepsilon(\one) \Id_{T(V)},$$
which is simply $\Id_{T(V)}$ if $\varepsilon$ is a Lie character.
\item
The pre-braiding $\sigma_{[,]}$ is natural with respect to the Lie unit $\one.$ Thus $\one$ is right $\sigma_{[,]}$-compatible with any $f \in V^*.$ 
\end{enumerate}
\end{lemma}

Everything is now ready for explicit calculations of differentials. Only the left ones give something interesting:

\begin{proposition}\label{thm:ComputLei}
Take a unital Leibniz algebra $(V,[,],\one)$ over $\k.$ The braiding $\sigma_{[,]}$ from \eqref{eqn:LeiBraid} and a braided character $\epsilon$ (for instance, a Lie character) define the following differential on $T(V)$:
\begin{align*}
{^{\epsilon}}\! d(v_1\ldots v_n) &= \epsilon(\one)\sum_{1\leqslant i < j \leqslant n} (-1)^{j-1}  v_1\ldots v_{i-1} [v_i,v_j] v_{i+1}\ldots\widehat{v_j}\ldots v_n + \\
&+\sum_{1\leqslant j \leqslant n}(-1)^{j-1}\epsilon (v_j) v_1\ldots \widehat{v_j}\ldots v_n.
\end{align*}
It comes, as usual, from the obvious pre-simplicial structure, completed into a weakly simplicial one by putting
$$s_{n;i}(v_1\ldots v_n) = \begin{cases}
 v_1\ldots v_{i-1}\one v_i\ldots v_n + v_1\ldots v_i \one v_{i+1}\ldots v_n & \text{if }v_i \in V', \\
 v_1\ldots v_{i-1}\one \one v_{i+1}\ldots v_n  & \text{if }v_i=\one.
\end{cases}$$

The complex $(T(V),{^{\epsilon}}\! d)$ is acyclic if $\epsilon(\one)=1.$
\end{proposition}

\begin{remark}\label{rmk:nonunitLei}
Like for associative algebras, in the \textit{non-unital case} one enriches $V$ with a formal Lie unit $\wV := V \oplus \k \one.$ The map $\sigma_{[,]}$ is a braiding \underline{on $\wV$} if and only if $[,]$ is Leibniz \underline{on $V$}. Taking the Lie character $\varepsilon (V)\equiv 0, \varepsilon(\one)=1$ on $\wV$ and restricting ${^{\varepsilon}}\! d$ to the subcomplex $T(V)\subset T(\wV),$ one recovers the familiar \textbf{\textit{Leibniz differential}}:
$${^{\varepsilon}}\! d(v_1\ldots v_n) = \sum_{1\leqslant i < j \leqslant n} (-1)^{j-1}  v_1\ldots v_{i-1} [v_i,v_j] v_{i+1}\ldots\widehat{v_j}\ldots v_n.$$
\end{remark}

\section{An upper world: categories}\label{sec:cat}

We now present a categorical version of our braided homology theory and of the pre-braidings and other ``braided'' ingredients for associative and Leibniz algebras. We work in the settings of a preadditive {strict} monoidal category, symmetric for Leibniz algebras. The world \textit{strict} is omitted but always implied here. This categorification is rather straightforward. A (more subtle and technical) categorical version of shelves and racks and of the corresponding braided differentials can be found in \cite{Lebed3}.

 Only the basic tools of category theory are used here; the famous books \cite{Cat} and \cite{Invariants} are excellent references for the general and, respectively, ``braided'' aspects of category theory. We also recommend the preprint \cite{Westrich}, where most of the categorical notions used here are nicely presented and illustrated. See also \cite{Lebed}.
 
Notations \eqref{eqn:phi_i} and \eqref{eqn:phi^i} are frequently used in this section.

\subsection{Categorifying braided differentials}\label{sec:cat_basic}

Start with a ``local'' categorical notion of pre-braiding:

\begin{definition}\label{def:WeakBr}
\begin{itemize}
 \item
An \emph{object} $V$ in a monoidal category $(\C,\otimes,\II)$ is called \emph{pre-braided} if it is endowed with a ``local'' \emph{pre-braiding}, i.e. a morphism $\sigma_{V} : V\otimes V \rightarrow V \otimes V$ satisfying (a categorical version of) the Yang-Baxter equation \eqref{eqn:YB}:
$$(\sigma_{V}\otimes \Id_V)\circ(\Id_V \otimes \sigma_{V})\circ(\sigma_{V}\otimes \Id_V) =(Id_V \otimes \sigma_{V}
)\circ(\sigma_{V}\otimes \Id_V)\circ(\Id_V \otimes \sigma_{V}).$$

\item One talks about a \emph{braided} object if $\sigma_{V}$ is an isomorphism.

\item A \emph{braided morphism} between pre-braided objects $(V, \sigma_V)$ and $(W, \sigma_W)$ in $\C$ is a morphism $f:V \rightarrow W$ respecting the pre-braidings:
\begin{equation}\label{eqn:Nat'}
(f\otimes f)\circ\sigma_V = \sigma_W\circ (f\otimes f):V\otimes V \rightarrow W\otimes W.
\end{equation}

\item A \emph{braided character} for a pre-braided object $(V,\sigma_{V})$ is a braided morphism to $(\II,\Id_\II),$ i.e.
$$(\epsilon\otimes \epsilon)\circ \sigma_{V} = \epsilon\otimes \epsilon :V\otimes V \rightarrow \II\otimes\II=\II.$$
\end{itemize}
\end{definition}

Every object in a (pre-)braided category $(\C,c)$ is (pre-)braided, with $\sigma_{V}=c_{V,V},$ since the YBE is automatic in $\C.$ However, the most interesting situation is that of a pre-braiding proper to an object. The idea of working with \textit{\textbf{``local''}} pre-braidings on $V$ instead of demanding the whole category $\C$ to be \textit{\textbf{``globally''}} braided is similar to what is done in \cite{Goyvaerts}, where self-invertible YB operators are considered in order to define YB-Lie algebras in an additive monoidal category $\C.$ Note that, contrary to their operator, our pre-braiding is \textbf{not necessarily invertible}.
 
Any pre-braided object $(V,\sigma)$ in a monoidal category comes with an action of the monoid $B_n^+$ on $V^{\otimes n}$ for each $n \ge 1,$ defined by formula \eqref{eqn:BnAction}. If the category is moreover preadditive, one can mimic the construction of the \textit{\textbf{quantum (co)shuffle (co)multiplication}} to get morphisms 
\begin{align*}\sh^{p,q}&:V^{\otimes n} = V^{\otimes p}\otimes V^{\otimes q}\rightarrow V^{\otimes n},\\
\csh^{p,q}&:V^{\otimes n} \rightarrow V^{\otimes p}\otimes V^{\otimes q}= V^{\otimes n}.
\end{align*}
Here $n=p+q.$ Still in the preadditive context, $-\sigma$ is well defined and gives a new pre-braiding for $V.$

\medskip

The definition \ref{def:BraidedCoalg} of (semi-)pre-braided coalgebra is directly transported to the setting of a monoidal category. A categorical version of definition \ref{def:NormCompat} is also straightforward:
 
\begin{definition}\label{def:NormCompatCat}
Take an object $V$ and two morphisms $\eta:\II \rightarrow V$ and $\psi:V \rightarrow \II$ in a monoidal category.
\begin{itemize}
\item
$\eta$ and $\psi$ are said to form a \emph{normalized pair} if $\psi \circ \eta = \Id_\II.$
\item $\eta$ and $\psi$ are called \emph{right $\sigma$-compatible} for a pre-braiding $\sigma$ on $V$ if 
$$(\Id_V \otimes \psi) \circ \sigma \circ (\Id_V \otimes \eta) = \eta \circ \psi : V  \rightarrow V.$$
\item A pre-braiding $\sigma$ on $V$ is called \emph{natural} with respect to $\eta$ if
\begin{align}
 \sigma \circ (\eta \otimes \Id_V) &= \Id_V \otimes \eta,\label{eqn:NatWRTSigmaC}\\
 \sigma \circ (\Id_V \otimes \eta) &= \eta \otimes \Id_V,\label{eqn:NatWRTSigma'C}
\end{align}
and \emph{semi-natural} (or \emph{demi-natural}) if only \eqref{eqn:NatWRTSigmaC} (resp. \eqref{eqn:NatWRTSigma'C}) holds.
\end{itemize}
\end{definition} 

\medskip
Further, recall the categorical notions of (bi)differentials:

\begin{definition}\label{def:DiffCat}
Take a preadditive category $\C.$
\begin{itemize}
\item  \emph{A degree $-1$ differential} for a family of objects  $\{V_n\}_{n\ge 0}$ in $\C$ is a family of morphisms $\{d_n:V_n\rightarrow V_{n-1}\}_{n>0},$ satisfying $d_{n-1}\circ d_n =0\; \forall n>1.$ 

\item \emph{A bidegree $-1$ bidifferential} is a pair of morphism families $\{d_n,d'_n:V_n \rightarrow V_{n-1}\}_{n>0}$  s.t.
\begin{equation}\label{eqn:bidiff}
d_{n-1} \circ d_n = d'_{n-1} \circ d'_n = d'_{n-1} \circ d_n +d_{n-1} \circ d'_n = 0 \qquad \forall n>1.
\end{equation}

\item Given a degree $-1$ differential $\{d_n\}_{n>0}$ for a family of objects  $\{V_n\}_{n\ge 0},$  one defines a \emph{contracting homotopy} as a family of morphisms $\{h_n:V_n\rightarrow V_{n+1}\}_{n\ge 0},$   satisfying
$$h_{n-1} \circ d_n + d_{n+1} \circ h_n = \Id_{V_n} \qquad \forall n>0.$$

\item Different types of \emph{simplicial objects} in $\C$ are defined by replacing the words ``vector space'' by ``object in $\C$'' in the definition \ref{def:Simpl}.

\item If $\C$ is moreover monoidal, \emph{a degree $-1$ tensor (bi)differential} for a $V \in \Ob(\C)$ is a degree $-1$ (bi)differential for the family of objects $\{V^{\otimes n}\}_{n\ge 0}.$ 

\end{itemize}
\end{definition}

Points \ref*{it:totaldiff} - \ref*{it:degen} of proposition \ref{thm:SimplBasics} remain valid for simplicial objects in a preadditive category $\C.$

The presence of a contracting homotopy means, in the category $\C = \kVect,$ that the complex $(V_n, d_n)$ is acyclic.

\medskip
With all the preparatory work above, theorems \ref{thm:cuts} and \ref{thm:BraidedSimplHom} and proposition \ref{thm:hom_operations} (with their proofs!) are generalized as follows (we freely use the notations from those theorems here):

\begin{theorem}\label{thm:BraidedSimplHomCat}
In a preadditive monoidal category  $(\C,\otimes,\II),$ take a pre-braided object $(V,\sigma)$ endowed with braided characters $\epsilon$ and $\zeta.$
\begin{enumerate}
\item 
 A bidegree $-1$ tensor bidifferential for $V$ can be defined by the morphism families
\begin{align*}
({^\epsilon}\! d)_n &:= \epsilon_1 \circ \Csh^{1,n-1},\\
(d^\zeta)_n &:= (-1)^{n-1}  \zeta_n \circ \Csh^{n-1,1}.
\end{align*}
\item\label{it:SimplCat}
This bidifferential can be refined into a pre-bisimplicial structure on 
 $(V^{\otimes n})_{n\ge 0}$ given by
\begin{align*}
d_{n;i}&:= \epsilon_1 \circ T_{p_{i,n}}^\sigma,\\
d'_{n;i}&:= \zeta_n \circ T_{p'_{i,n}}^\sigma.
\end{align*}

\item If a comultiplication $\Delta$ endows $(V,\sigma)$ with a pre-braided coalgebra structure, then the morphisms $s_{n;i} := \Delta_i$ complete the preceding structure into a very weakly bisimplicial one.

\item If $(V,\sigma,\Delta)$ is a semi-pre-braided coalgebra, then one obtains a very weakly simplicial object $(V^{\otimes n},d_{n;i},s_{n;i})$ only.

\item If $\Delta$ is moreover $\sigma$-cocommutative, then the above structures on $(V^{\otimes n})_{n\ge 0}$ are weakly (bi)simplicial.

\item Take a morphism $\eta: \II \rightarrow V.$ The family $h_{n}:= (-1)^n \Id_n \otimes \eta:  V^{\otimes n} \rightarrow V^{\otimes (n+1)}$
\begin{itemize}
\item is a contracting homotopy for $(V^{\otimes n},({^\epsilon}\! d)_n)$ if the pair $(\eta,\epsilon)$ is normalized, and $\sigma$ is demi-natural with respect to $\eta.$

\item is a contracting homotopy for $(V^{\otimes n},(d^\zeta)_n)$ if the pair $(\eta,\zeta)$ is normalized and right $\sigma$-compatible.
\end{itemize}
\end{enumerate}
\end{theorem}

\subsection{Basic examples revisited}\label{sec:cat_ex}

Recall the categorical definitions of associative and Leibniz algebras:

\begin{definition}\label{def:CatUAA}
\begin{itemize}
\item 
A \emph{(unital) associative algebra} in a monoidal category $\C,$ abbreviated as (U)AA, is an object $V$ together with morphisms $\mu:V\otimes V\rightarrow V$ (and $\nu:\II\rightarrow V$), satisfying  the associativity (and the unit) conditions:
 \begin{align*}
  \mu\circ(\Id_V\otimes\mu) &=\mu\circ(\mu \otimes \Id_V):V^{\otimes 3}\rightarrow V,\\
  \mu\circ(\Id_V\otimes\nu) &=\mu\circ(\nu \otimes \Id_V) = \Id_V.
 \end{align*}
 For a \emph{right-unital associative algebra} we demand only the 
 $\mu\circ(\Id_V\otimes\nu) = \Id_V$
 part of the last condition.
 
\item  The subcategory of UAAs and unital algebra morphisms (i.e. morphisms respecting $\mu$ and $\nu$) in $\C$ is denoted by $\Alg(\C),$ or $\mathbf{Alg}(\C)$ in the non-unital case.

\item An \emph{algebra character} for $V \in \Alg(\C)$ is a unital algebra morphism $\epsilon \in \Hom_{\Alg(\C)}(V,\II),$ with the \emph{default UAA structure on $\II$} ($\mu=\nu=\Id_\II$). A \emph{non-unital algebra character} is a morphism in $\mathbf{Alg}(\C)$ only.
\end{itemize}
 \end{definition}

 \begin{definition}\label{def:CatULA}
 \begin{itemize}
\item 
  A \emph{(unital) Leibniz algebra} in a symmetric preadditive category $\C,$ abbreviated as (U)LA, is an object $V$ together with morphisms $[,]:V\otimes V\rightarrow V$ (and $\nu:\II\rightarrow V$) satisfying the generalized Leibniz (and the Lie unit) conditions:
$$ [,]\circ(\Id_V\otimes [,]) =[,]\circ([,] \otimes \Id_V)-[,]\circ([,] \otimes \Id_V)\circ(\Id_V\otimes c_{V,V}) :V^{\otimes 3}\rightarrow V,$$
$$ [,] \circ(\Id_V\otimes\nu) = [,] \circ(\nu \otimes \Id_V) =0 :V\rightarrow V.$$

\item The subcategory of ULAs and unital Leibniz algebra morphisms (i.e. morphisms respecting $[,]$ and $\nu$) in $\C$ is denoted by  $\Lei(\C),$ or $\mathbf{Lei}(\C)$ in the non-unital case.

\item \emph{A Lie character} for $V \in \Lei(\C)$ is an $\epsilon \in \Hom_{\Lei(\C)}(V,\II),$ with the \emph{default ULA structure on $\II$} (zero bracket and $\nu = \Id_\II$).  A \emph{non-unital Lie character} is an $\epsilon \in \Hom_{\mathbf{Lei}(\C)}(V,\II).$
\end{itemize}
\end{definition}

See for instance \cite{Baez} and \cite{MajidBraided} for the definition of algebras in a monoidal category, and \cite{Goyvaerts} for a survey on categorical Lie algebras. Note that $\Alg(\Vect)$ and $\Lei(\Vect)$ are the familiar categories of $\k$-linear unital associative and Leibniz algebras respectively.

\begin{definition}\label{def:NormalMor}
\emph{A normalized morphism} $\varphi:V \rightarrow W$ for $V,W \in \Alg(\C)$ or $\in \Lei(\C)$ is a morphism in $\C$ respecting the units, i.e. $\varphi \circ \nu_V = \nu_W.$
\end{definition}

For $W = \II$ this means that $(\nu_V,\varphi)$ is a normalized pair.

\medskip
Now, the content of subsections \ref{sec:bar} and \ref{sec:Lei} can be categorified as follows:

\begin{theorem}\label{thm:Bar&Lei}
\begin{enumerate}
\item Take a right-unital associative algebra $(V,\mu,\nu)$ in a monoidal category $(\C,\otimes,\II).$ 
\begin{enumerate}
\item $V$ can be endowed with a pre-braiding
\begin{equation}\label{eqn:sigmaUAA}
\sigma_{Ass} := \nu \otimes \mu : V\otimes V = \II \otimes V \otimes V \rightarrow V \otimes V.
\end{equation} 
\item If $\nu$ is moreover a two-sided unit, then the comultiplication 
$$\Delta_{Ass} := \nu \otimes \Id_V:V = \II \otimes V \rightarrow V \otimes V$$ 
completes $\sigma_{Ass}$ into a $\sigma_{Ass}$-cocommutative pre-braided coalgebra structure.
\item Any algebra character $\epsilon \in \Hom_{\Alg(\C)}(V,\II)$ is a braided character for $(V,\sigma_{Ass}).$
\item The pre-braiding $\sigma_{Ass}$ is demi-natural with respect to the unit $\nu.$ Moreover, for any algebra character $\epsilon,$ the pair $(\nu,\epsilon)$ is normalized. 
\end{enumerate}
\item Take a unital Leibniz algebra $(V,[,],\nu)$ in a symmetric preadditive category $(\C,\otimes,\II,c).$ 
\begin{enumerate}
\item $V$ can be endowed with an invertible braiding 
\begin{equation}\label{eqn:sigmaULA}
\sigma_{Lei}:= c_{V,V}+ \nu \otimes [,].
\end{equation}
\item If $\C$ is additive and if one has a Leibniz algebra decomposition $V\simeq V' \oplus \II,$ then 
\begin{align*}
\Delta_{Lei}|_{V'} &:= \nu \otimes \Id_{V'} + \Id_{V'} \otimes \nu :V' \rightarrow V \otimes V,\\
\Delta_{Lei}|_{\II} &:= \nu \otimes \nu : \II \rightarrow V \otimes V
\end{align*}
completes $\sigma_{Lei}$ into a $\sigma_{Lei}$-cocommutative semi-braided coalgebra structure. 
\item Any Lie character $\epsilon \in \Hom_{\Lei(\C)}(V,\II)$ is a braided character for $(V,\sigma_{Lei}).$
\item The braiding $\sigma_{Lei}$ is natural with respect to the unit $\nu.$ Moreover, for any Lie character $\epsilon,$ the pair $(\nu,\epsilon)$ is normalized. 
\end{enumerate}
\end{enumerate}
\end{theorem}

Observe that in the Leibniz algebra setting, the naturality (with respect to morphisms $\nu$ and $[,]$ in particular) and the symmetry of the braiding $c$ are essential in proving that $\sigma_{Lei}$ is indeed a  braiding, while the naturality of $c$ with respect to $\epsilon$ shows that $\epsilon$ is a braided character for $(V,c_{V,V})$ (which implies that it is a braided character for $(V,\sigma_{V})$ if it preserves the Leibniz structure). 

\begin{remark}
According to the theorem, a ULA $V$ is a \textit{{``doubly braided''}} object: $\sigma_{V}$ and $c_{V,V}$ are indeed two distinct braidings on $V.$ One can say more: they endow the tensor powers of $V$ with an action of the \textit{{virtual braid group}} (\cite{KauffmanVirtual}, \cite{Ver}). The close connections between (pre-)braided objects and virtual braid groups are studied in detail in \cite{Lebed3}.
\end{remark}

Theorems \ref{thm:BraidedSimplHomCat} and \ref{thm:Bar&Lei} put together give categorical versions of propositions \ref{thm:ComputAss} and \ref{thm:ComputLei}, as well as of the ``non-unital'' remarks \ref{rmk:nonunit} and \ref{rmk:nonunitLei}:
\begin{corollary}\label{crl:cat}
\begin{enumerate}
\item Any algebra character $\epsilon: V \rightarrow \II$ for a UAA $(V,\mu,\nu)$ in a preadditive monoidal category $\C$ produces a degree $-1$ tensor differential 
$$({^\epsilon}\! d)_n := \epsilon_1 +\sum_{i=1}^{n-1}(-1)^i \mu_i,$$
for $V,$ with a contracting homotopy $h_{n}= (-1)^n \nu_{n+1}.$ Any non-unital algebra characters $\epsilon,\zeta$ for an associative algebra $(V,\mu)$ in $\C$ produce a degree $-1$ tensor differential
$$({^\epsilon}\! d^\zeta)_n := \epsilon_1 +\sum_{i=1}^{n-1}(-1)^i \mu_i +(-1)^n \zeta_n.$$

\item Any Lie character $\epsilon: V \rightarrow \II$ for a ULA $(V,[,],\nu)$ in a symmetric preadditive category $\C$ produces a degree $-1$ tensor differential 
$$ ({^\epsilon}\! d)_n := \epsilon_1\circ \underset{-c}{\cshuffle}^{1,n-1} +\sum_{1\leqslant i < j \leqslant n} (-1)^{j-1} [,]_i\circ (\Id_i \otimes c_{V^{j-i-1},V} \otimes \Id_{n-j}),$$
for $V,$  with a contracting homotopy $h_{n}= (-1)^n \nu_{n+1}.$ Any non-unital Lie character $\epsilon$ for a Leibniz algebra $(V,[,])$ in a symmetric additive $\C$ produces a degree $-1$ tensor differential
$$({^\epsilon}\! d)_n := \sum_{1\leqslant i < j \leqslant n} (-1)^{j-1} [,]_i\circ (\Id_i \otimes c_{V^{j-i-1},V} \otimes \Id_{n-j}).$$
\end{enumerate}
\end{corollary}

Working in the category $\Vect$ in section \ref{sec:examples}, we noticed that the pre-braidings obtained for associative and Leibniz algebras \textbf{{encode}} the underlying algebraic structures (cf. lemmas \ref{thm:Bar} and \ref{thm:Lei}). It is still true, with some additional technical assumptions, in the categorical setting:

\begin{lemma}\label{thm:encode_cat} 
\begin{enumerate}
\item Take an object $V$ in a monoidal category $(\C,\otimes,\II)$ endowed with two morphisms $\mu: V \otimes V \rightarrow V$ and $\nu:\II \rightarrow V,$ with $\nu$ being a two-sided unit for $\mu.$ The morphism $\sigma_{Ass}$ defined by \eqref{eqn:sigmaUAA} is a pre-braiding {if and only if} $\mu$ is associative.
\item Take an object $V$ in a symmetric preadditive category $(\C,\otimes,\II,c)$ endowed with two morphisms $[,]: V \otimes V \rightarrow V $ and $\nu:\II \rightarrow V,$ with $\nu$ being a Lie unit for $[,].$ Additionally suppose the existence of a {normalized} morphism $\gamma:V\rightarrow\II.$ The morphism $\sigma_{Lei}$ defined by \eqref{eqn:sigmaULA} is a braiding {if and only if} $[,]$ satisfies the Leibniz condition.
\end{enumerate}
\end{lemma}

\begin{proof}
One repeats the proofs of lemmas \ref{thm:Bar} and \ref{thm:Lei}. The only non-trivial step is to show that
$$\nu\otimes\nu\otimes f = \nu\otimes\nu\otimes g:V^{\otimes 3}\rightarrow V^{\otimes 3}$$ 
implies
$$f=g:V^{\otimes 3}\rightarrow V.$$

When $\nu$ is a left unit for $\mu,$ this is done by applying $\mu\circ(\Id_V\otimes \mu)$ to both sides of the first identity. In the Leibniz case, apply $\gamma\otimes\gamma\otimes \Id_V.$
\end{proof}

Similar considerations lead to an ``if and only if'' type result for morphisms:

\begin{lemma}\label{prop:braid_nat}
\begin{itemize}
\item In the settings of theorem \ref{thm:Bar&Lei}, any morphism $f:V\rightarrow W$ in $\Alg(\C)$ (resp. $\Lei(\C)$) is braided, with $V$ and $W$ endowed with the pre-braidings $\sigma_{Ass}$ (resp. $\sigma_{Lei}$).
\item Suppose additionally, for the algebra case, that $\nu$ is a two-sided unit, and, for the Leibniz case, the existence of a normalized morphism $\gamma:V\rightarrow\II.$ Then any braided normalized (cf. definition \ref{def:NormalMor}) morphism $f:V\rightarrow W$  in $\C$ necessarily respects the multiplications, i.e. remains a morphism in $\Alg(\C)$ (resp. $\Lei(\C)$).
\end{itemize}
\end{lemma}

Thus an associative/Leibniz algebra morphism is the same thing (modulo some technical conditions) as a  normalized braided morphism for the corresponding ``structural'' pre-braiding, illustrating the precision with which our pre-braidings capture the algebraic properties of the corresponding structure. 

This result can be compactly restated as follows. Denote by $\Br^\nu(\C)$ the subcategory of a monoidal category $\C$ with
\begin{itemize}
\item as objects, pre-braided objects $V$ endowed with a ``unit'' $\nu:\II \rightarrow V$;
\item as morphisms, braided normalized morphisms in $\C.$ 
\end{itemize}

\begin{proposition}
\begin{enumerate}
\item Given a monoidal category $\C,$ one can see $\Alg(\C)$ as a full subcategory of $\Br^\nu(\C)$ via  the functor
\begin{align*}
\Alg(\C) &\hookrightarrow \Br^\nu(\C),\\
(V,\mu,\nu) &\mapsto (V,\sigma_{Ass},\nu),\\
(f:V\rightarrow W) &\mapsto (f:V\rightarrow W).
\end{align*}
\item Given a symmetric preadditive category $\C,$ one can see $\Lei^*(\C)$ (the category of ULAs $V$ in $\C$ endowed with a normalized morphism $\gamma:V\rightarrow\II$) as a full subcategory of $\Br^\nu(\C)$ via 
\begin{align*}
\Lei^*(\C) &\hookrightarrow \Br^\nu(\C),\\
(V,[,],\nu) &\mapsto (V,\sigma_{Lei},\nu),\\
(f:V\rightarrow W) &\mapsto (f:V\rightarrow W).
\end{align*}
\end{enumerate}
\end{proposition}

\subsection{The super trick}

The first bonus one generally gains when passing to abstract symmetric categories is the possibility to derive graded and super versions of algebraic results for free, thanks to the Koszul flip $\tau_{Koszul}$ from \eqref{eqn:KoszulFlip}. One clearly sees where to put signs, which is otherwise quite difficult to guess. Here is a typical example.

Take a \textbf{\textit{graded unital Leibniz algebra}} $(V,[,],\nu),$ i.e. an object of $\Lei$  $(\VectGrad),$ where $\VectGrad$ is the usual additive monoidal category of graded $\k$-vector spaces, endowed with the symmetric braiding $\tau_{Koszul}.$ Leibniz condition in this setting is
$$[v,[w,u]]=[[v,w],u]-(-1)^{\deg u \deg w}[[v,u],w]$$
for any homogeneous elements $v,w,u \in V.$ Cf. fig. \ref{pic:Lei} illustrating \eqref{eqn:Lei}, with the crossing on the right corresponding to the ``internal'' braiding $c_{V,V}=\tau_{Koszul}.$

Theorem \ref{thm:Bar&Lei} gives a braiding for $V$:
$$ \sigma_{V} :v\otimes w \longmapsto (-1)^{\deg v \deg w} w \otimes v + \one \otimes [v,w],$$
which, together with a Lie character $\epsilon:V_0\rightarrow \k,$ can be fed into the machinery from theorem \ref{thm:BraidedSimplHomCat}:
\begin{proposition}
\begin{enumerate}
\item 
A $\k$-linear graded unital Leibniz algebra $(V,[,],\nu)$ with a Lie character $\epsilon$ can be endowed with the degree $-1$ tensor differential
\begin{align*}
{^{\epsilon}}\! d(v_1\ldots v_n) &=\sum_{1\leqslant i < j \leqslant n} (-1)^{j-1+\alpha_{i,j}}  v_1\ldots v_{i-1} [v_i,v_j] v_{i+1}\ldots\widehat{v_j}\ldots v_n + \\
&+\sum_{1\leqslant j \leqslant n}(-1)^{j-1+\alpha_{0,j}}\epsilon (v_j) v_1\ldots \widehat{v_j}\ldots v_n,
\end{align*}
where $\alpha_{i,j}:=\deg(v_j)\sum_{i < k< j}\deg(v_k).$
\item
A $\k$-linear graded Leibniz algebra $(V,[,])$ with a non-unital Lie character $\epsilon$ can be endowed with the degree $-1$ tensor differential
$$ {^{\epsilon}}\! d(v_1\ldots v_n) =\sum_{1\leqslant i < j \leqslant n} (-1)^{j-1+\alpha_{i,j}}  v_1\ldots v_{i-1} [v_i,v_j] v_{i+1}\ldots\widehat{v_j}\ldots v_n.$$
\end{enumerate}
All the $v_i$'s are taken homogeneous here.
\end{proposition}
Observe that the $(-1)^{\alpha_{i,j}}$ part of the sign comes from the Koszul braiding, while $(-1)^{j-1}$ appears because we take the opposite braiding when defining $({^\epsilon}\! d)_n := \epsilon_1 \circ \Csh^{1,n-1}.$ 

\medskip
{Leibniz superalgebras} are treated similarly: one has just to work in the category of super vector spaces over $\k$ (cf. \cite{LeiSuper}). One thus recovers the \textit{\textbf{Leibniz superalgebra homology}}, which is a lift of the Lie superalgebra homology.

Similarly, one gets for free the \textit{\textbf{color Leibniz algebra homology}} (cf. \cite{LeiColor}, or \cite{LieColor} for a Lie version). Concretely, take a finite abelian group $\Gamma$ endowed with an antisymmetric bicharacter $\chi.$ The category $_\Gamma\!\Vect$ of $\k$-vector spaces graded over $\Gamma$ is symmetric additive, with the usual $\Gamma$-graded tensor product, the zero-graded $\k$ as its identity object and, as a braiding, the \emph{color flip}
$$\tau_{color} : v\otimes w\longmapsto \chi(f,g) w\otimes v$$
for homogeneous $v$ and $w$ graded over $f$ and $g \in \Gamma$ respectively. Then color Leibniz algebras are precisely Leibniz algebras in $_\Gamma\!\Vect$

See also \cite{Westrich} for an excellent survey of different types of braided Lie algebras.

\subsection{Co-world, or the world upside down}\label{sec:co}

One more nice feature of the categorical approach is an automatic treatment of \textit{\textbf{dualities}}. The most common notion of duality, the ``upside-down'' one, is described here, with the cobar complex for coalgebras (cf. \cite{Cartier}, \cite{Doi}) providing an example. In the monoidal context, one has two more dualities, the ``right-left'' and the combined ones, treated in the next subsection.

\begin{definition}
Given a category $\C,$ its \emph{dual (or opposite) category} $\C^{\op}$ is constructed by keeping the objects of $\C$ and reversing all the arrows. One writes $f^{\op} \in \Hom_{\C^{\op}}(W,V)$ for the morphism in $\C^{\op}$ corresponding to an $f\in \Hom_{\C}(V,W).$
\end{definition}

We sometimes call $\C^{\op}$ a \emph{co-category} in order to avoid confusion with other notions of duality. Observe that this construction is involutive: $(\C^{\op})^{\op} = \C.$ 

The \textbf{\textit{duality principle}} (cf. \cite{Cat}, section II.2) tells that a ``categorical'' \underline{theorem} for $\C$ implies a dual theorem for $\C^{\op}$ by reversing all the arrows and the order of arrows in every composition. Our aim here is to apply this principle to theorems \ref{thm:BraidedSimplHomCat} and \ref{thm:Bar&Lei}.

To get a notion of duality for categorical \underline{structures}, it suffices to place them into the co-category. For example, a \emph{counital co-Leibniz coalgebra} (= co-ULA) in a symmetric preadditive category $\C$ is an object $V$ together with morphisms $\partial:V\rightarrow V\otimes V$ and $\varepsilon: V \rightarrow \II,$
such that $(V,\partial^{\op},\varepsilon^{\op})$ is a ULA in $\C^{\op}.$ In other words, $\partial$ and $\varepsilon$ satisfy
\begin{align}
(\Id_V\otimes\partial)\circ\partial &= (\partial\otimes\Id_V)\circ\partial - (\Id\otimes c_{V,V})\circ(\partial\otimes\Id)\circ\partial,\label{eqn:coLei}\\
(\Id_V\otimes\varepsilon)\circ\partial &= (\varepsilon\otimes\Id_V)\circ\partial=0.\notag
\end{align}

A convenient way to handle the ``upside-down'' duality is the graphical one: changing from $\C$ to $\C^{\op}$ consists simply in {turning all the diagrams upside down}, i.e. taking a \textit{\textbf{horizontal mirror image}}. Here is an example for the {co-Leibniz} condition \eqref{eqn:coLei} (cf. fig. \ref{pic:Lei}):
\begin{center}
\begin{tikzpicture}[scale=0.2]
\draw (0,0)--(2,-2);
\draw (4,0)--(2,-2);
\draw (3,-1)--(2,0);
\draw (2,-3)--(2,-2);
\fill[teal] (3,-1) circle (0.3);
\fill[teal] (2,-2) circle (0.3);
\node at (5,-1.5) {$=$};
\node at (6,-1.5) {};
\end{tikzpicture}
\begin{tikzpicture}[scale=0.2]
\draw (0,0)--(2,-2);
\draw (4,0)--(2,-2);
\draw (1,-1)--(2,0);
\draw (2,-3)--(2,-2);
\fill[teal] (1,-1) circle (0.3);
\fill[teal] (2,-2) circle (0.3);
\node at (5,-1.5) {$-$};
\node at (6,-1.5) {};
\end{tikzpicture}
\begin{tikzpicture}[scale=0.2]
\draw (0,0)--(2,-2);
\draw (4,0)--(1,-1);
\draw (2,-2)--(2,0);
\draw (2,-3)--(2,-2);
\fill[teal] (1,-1) circle (0.3);
\fill[teal] (2,-2) circle (0.3);
\node at (5,-3) {.};
\end{tikzpicture}
   \captionof{figure}{Co-Leibniz condition }\label{pic:CoLei}
\end{center}

\medskip %
We now make a list of dualities for the categorical structures relevant for this paper: 
 
\begin{center}
\begin{tabular}{|c|c|}
\hline 
 unital associative algebra $(V,\mu,\nu)$ & co-UAA $(V,\mu^{\op},\nu^{\op})$\\
\hline
 unital Leibniz algebra $(V,[,],\nu)$ & co-ULA $(V,[,]^{\op},\nu^{\op})$\\
\hline
 algebra character $\varphi$ for $(V,\mu,\nu)$ & coalgebra co-character $\varphi^{\op}$ for  $(V,\mu^{\op},\nu^{\op})$\\
 \hline
Lie character $\varphi$ for $(V,[,],\nu)$ & co-Lie co-character $\varphi^{\op}$ for  $(V,[,]^{\op},\nu^{\op})$\\
\hline
pre-braiding $\sigma$ for $V$ & pre-braiding $\sigma^{\op}$ for $V$\\
\hline
braided character $\epsilon$ for $(V,\sigma)$ & braided co-character $\epsilon^{\op}$ for $(V,\sigma^{\op})$\\ 
\hline
(bi)degree $-1$ tensor (bi)bidifferential for $V$ & (bi)degree $1$ tensor (bi)differential for $V$\\
\hline
\end{tabular}
   \captionof{table}{Categorical duality for structures}\label{tab:CatDual}
\end{center}

The subcategory of co-UAAs and co-ULAs in $\C$ are denoted by $\coAlg(\C)$ and $\coLei(\C)$ respectively.
 
 \begin{remark}\label{rmk:DualityForShuffle}
For a pre-braided object $(V,\sigma)$ and the action \eqref{eqn:BnAction} of $B_n^+$ on $V^{\otimes n},$ one has 
$$(T_s^{\sigma})^{\op} = T_{s^{-1}}^{(\sigma^{\op})} \in End_{\C^{\op}}(V^{\otimes n}) \qquad \forall s \in S_n.$$
 Thus, assuming the category preadditive, the definition \eqref{eqn:cosh} of \textit{\textbf{quantum co-shuffle comultiplication}} is translated as $\csh^{p,q}=({\underset{\sigma^{\op}}{\shuffle}}^{p,q})^{\op},$ automatically giving all the properties of this structure.
\end{remark} 

 \medskip
Everything is now ready for dualizing theorems \ref{thm:BraidedSimplHomCat} and \ref{thm:Bar&Lei}. We present only short versions of these results here, leaving the dualization of the points concerning simplicial and pre-braided coalgebra structures to the reader.
 
 \medskip %
\medskip
\textbf{Theorem \ref*{thm:BraidedSimplHomCat}$^\mathbf{co}$.}
\textit{
Let $\C$ be a preadditive monoidal category. For any pre-braided object $(V,\sigma)$ with braided co-characters $e$ and $c,$ a bidegree $1$ tensor bidifferential for $V$ can be defined by
\begin{align*}
({_e}\! d)^n &:=\Sh^{1,n} \circ (e \otimes \Id_n),\\
(d_c)^n &:=(-1)^n \Sh^{n,1} \circ (\Id_n \otimes c).
\end{align*} 
}

\bigskip\bigskip %
\medskip
\textbf{Theorem \ref*{thm:Bar&Lei}$^\mathbf{co}$.}
\textit{
\begin{enumerate}
\item Take a counital coassociative coalgebra $(V,\Delta,\varepsilon)$ in a monoidal category $\C.$
\begin{enumerate}
\item $V$ can be endowed with a pre-braiding $\sigma_{coAss} := \varepsilon \otimes \Delta.$
\item Any coalgebra co-character $e\in \Hom_{\coAlg(\C)}(\II,V)$ is a braided co-character for $(V,\sigma_{coAss}).$
\end{enumerate}
\item Take a counital co-Leibniz coalgebra $(V,\partial,\varepsilon)$ in a symmetric preadditive category $(\C,\otimes,\II,c).$ 
\begin{enumerate}
\item $V$ can be endowed with a braiding $\sigma_{coLei}:= c_{V,V}+ \varepsilon \otimes \partial.$
\item Any co-Lie co-character $e\in \Hom_{\coLei(\C)}(\II,V)$ is a braided co-character for $(V,\sigma_{coLei}).$
\end{enumerate}
\end{enumerate}
}
\medskip

A graphical depiction of, for instance, $\sigma_{coAss}$ is by construction the horizontal mirror image of the diagram one had for UAAs:
  \begin{center}
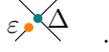

\begin{tikzpicture}[scale=0.6]
 \draw (0,0) -- (1,-1);
 \draw (1,0) -- (0.5,-0.5);
 \draw (0.3,-0.7) -- (0,-1);
 \node at (0.5,-0.5) [right]{$\Delta$};
 \node at (0.3,-0.7) [left]{$\varepsilon$};
 \fill[orange] (0.3,-0.7) circle (0.1);
 \fill[teal] (0.5,-0.5) circle (0.1);
 \node at (2,-1) {.};
\end{tikzpicture}
   \captionof{figure}{$\sigma_{coAss}=HorMirror(\sigma_{Ass})$}\label{pic:BrCoUAA}
\end{center} 

A co-version of corollary \ref{crl:cat} is then formulated in the evident way, with dual explicit formulas. ``If and only if'' lemmas \ref{thm:encode_cat} and \ref{prop:braid_nat} are also dualized directly. In particular, the pre-braidings from the previous theorem encode the co-associativity (resp. co-Leibniz) condition.

\medskip

We finish this section with some remarks proper to our favorite category $\Vect$ (everything remaining valid, as usual, in $\kVect$). 

\begin{lemma} 
In $\Vect,$ a map $e:\k\rightarrow V, \alpha\mapsto \alpha \mathbf{e}$ for a co-UAA $(V,\Delta,\epsilon)$ is a non-unital coalgebra co-character if and only if $\mathbf{e} \in V$ is {\emph{group-like}}, i.e. $\Delta(\mathbf{e})=\mathbf{e}\otimes \mathbf{e}.$ A non-unital Lie co-character for a co-ULA $(V,\partial,\epsilon)$ corresponds to an $\mathbf{e} \in \Ker(\partial).$ 
\end{lemma} 

Further, ``{non-unital}'' remarks \ref{rmk:nonunit} and \ref{rmk:nonunitLei} admit co-versions. To create a counit for a coassociative or co-Leibniz coalgebra $(V,\delta)$ (resp. $(V,\partial)$), one extends it by adding a formal element $\wV:=V\oplus \k\one,$ modifying the comultiplication:
\begin{align*}
\Delta(v) &= \delta(v)+\one \otimes v+ v \otimes \one  \qquad \forall v \in V,\\
\Delta(\one) &=\one \otimes \one
\end{align*}
in the coassociative coalgebra case, and 
$$\partial (\one)=0,$$
keeping the original $\partial$ on $V,$ in the co-Leibniz case. The pre-braiding $\sigma_{coAss}$ (resp. $\sigma_{coLei}$) on $\wV$ now characterizes the coassociativity (resp. co-Leibniz) condition for $V.$ Further, the application $\varepsilon\in (\wV)^*$ given by $\varepsilon(V)\equiv 0, \varepsilon (\one)=1$ is a coalgebra (resp. co-Lie) counit for $\Delta$ (resp. $\partial$), and $\one$ is a  group-like element (resp. $\one \in \Ker(\partial)$). 

The left braided differentials obtained in this setting for coalgebras are described in

\begin{proposition}\label{thm:cobar}
Given a $\k$-linear coalgebra $(V,\delta),$ extend it to a counital one $(\wV, \Delta, \varepsilon)$ as  above. Then the group-like $\one$ gives, via theorem \ref*{thm:BraidedSimplHomCat}$^\mathbf{co}$, the following differential on $T(\wV)$:
$${_{\one}}\! d(v_1\ldots v_n) =\one v_1\ldots v_n +\sum_{i=1}^{n} (-1)^{i}v_1\ldots v_{i-1}\Delta(v_i)v_{i+1}\ldots v_n.$$
The ideal $I_\one$ of the tensor algebra $T(\wV)$ generated by the element $\one$ is ${_{\one}}\! d$-stable.
The differential induced on $T(\wV)/I_\one \simeq T(\wV/\k\one) \simeq T(V)$ is
$$\widetilde{{_{\one}}\! d}(v_1\ldots v_n) = \sum_{i=1}^{n} (-1)^{i}v_1\ldots v_{i-1}\delta(v_i)v_{i+1}\ldots v_n.$$
\end{proposition}

One eagerly recognizes the \textbf{\textit{cobar differential}} for coalgebras.

\subsection{Right-left duality}\label{sec:r-l}

One more notion of duality is available for a monoidal category $(\C,\otimes,\II).$ One can simply change its tensor product to the opposite one: $ V \opp W := W \otimes V$ for objects, and similarly for morphisms.
We call this new monoidal category \textbf{\textit{monoidally dual}} to $\C,$ denoting it by $\C^{\opp}$ (there seem to be no universally accepted notation, some authors even using $\C^{\op}$ here and another notation for co-categories). Graphically, the categories $\C$ and $\C^{\opp}$ differ by the {\textit{vertical mirror symmetry}} for all diagrams.

Applying monoidal duality to a co-category $\C^{\op},$ one gets $\C^{\oppp}:=(\C^{\op})^{\opp} \simeq (\C^{\opp})^{\op}.$ Graphically, it corresponds to the {\textit{central symmetry}}.

Similarly to what we have seen for $\C^{\op},$ all ``categorical'' notions and theorems have monoidally dual versions in $\C^{\opp}.$ This gives in particular \textbf{\textit{right differentials}}
$(d^\epsilon)_n,(d_e)^n,$ monoidally dual to the left ones $({^\epsilon}\! d)_n,({_e}\! d)^n.$ Note that these differentials should be endowed with a sign (cf. theorem \ref{thm:cuts}) if one wants a bidifferential structure.

One also has \textbf{\textit{right braidings}}, monoidally dual to those from theorems \ref*{thm:Bar&Lei} and \ref*{thm:Bar&Lei}$^\mathbf{co}$. In particular, a new braiding emerges for UAAs:
$$\sigma_{Ass}^{r} := \mu \otimes \nu = VertMirror(\sigma_{Ass}) =
\begin{tikzpicture}[scale=0.5]
 \node at (-0.5,0.5) {};
 \draw  (0,0) -- (0.5,0.5);
 \draw  (1,0) -- (0,1);
 \draw  (0.7,0.7) -- (1,1);
 \node at (0.5,0.5) [left]{$\mu$};
 \node at (0.7,0.7) [right] {$\nu$}; 
 \fill[orange] (0.7,0.7) circle (0.1);
 \fill[teal] (0.5,0.5) circle (0.1);
 \node at (2,0.2) {.};
\end{tikzpicture}$$

Remark that the Leibniz algebra structure is not right-left symmetric: a Leibniz algebra in $\C^{\opp}$ is in fact a \textbf{\textit{left Leibniz algebra}} in $\C$ (cf. \cite{LoPi}). Thus one automatically obtains braided homology theories for left Leibniz algebras.

\section{Braided modules and homologies with coefficients}\label{sec:coeffs}

We introduce here the notion of (bi)modules over a pre-braided object $V$ in a monoidal category. These ``braided'' modules generalize, in quite an unexpected manner, the usual notions of (bi)modules for algebraic structures. Since at the same time a braided module generalizes a braided character, one naturally arrives to homologies of pre-braided objects with coefficients. As particular cases, we point out Hochschild and Chevalley-Eilenberg complexes. 

Fix a \underline{monoidal category} $(\C,\otimes,\II).$

\begin{definition}
\begin{itemize}
\item A \emph{right module} over a pre-braided object $(V, \sigma)$ is an object $M \in \Ob (\C)$ equipped with a morphism $\rho:M\otimes V \rightarrow M$ satisfying
\begin{equation}\label{eqn:BrMod}
\rho \circ (\rho \otimes \Id_V)=\rho \circ (\rho \otimes \Id_V)\circ (\Id_M\otimes \sigma):M\otimes V\otimes V\rightarrow M,
\end{equation}
\begin{center}
\begin{tikzpicture}[scale=0.3]
 \draw[olive,ultra thick] (0,0) -- (0,3);
 \draw (1,0) -- (0,1);
 \draw (2,0) -- (0,2);
 \node at (0,2) [left]{$\rho$};
 \node at (0,1) [left]{$\rho$};
 \fill[teal] (0,2) circle (0.2);
 \fill[teal] (0,1) circle (0.2); 
 \node at (2,0) [below] {$V$};
 \node at (1,0) [below] {$V$};
 \node at (0,0) [below] {$M$};
 \node at (0,3) [above] {$M$};
\node  at (4,1.5){$=$};
\end{tikzpicture}
\begin{tikzpicture}[scale=0.3]
\node  at (-2,1.5){};
 \draw[olive,ultra thick] (0,0) -- (0,3);
 \draw (1,0) -- (7/8,0.25);
 \draw (0.5,1) -- (0,2);
 \draw (2,0) -- (0,1);
 \node at (0,2) [left]{$\rho$};
 \node at (0,1) [left]{$\rho$};
 \fill[teal] (0,2) circle (0.2);
 \fill[teal] (0,1) circle (0.2); 
 \node at (2,0) [below] {$V$};
 \node at (1,0) [below] {$V$};
 \node at (0,0) [below] {$M$};
 \node at (0,3) [above] {$M$};
 \node at (1.2,1.2) {$\sigma$};
 \node at (3,0) {.};
\end{tikzpicture}
   \captionof{figure}{Braided module}\label{pic:BrMod}
\end{center}
We talk about \emph{braided $V$-modules} when the pre-braiding $\sigma$ is clear from the context.
\item A \emph{left module} is a right one in $\C^{\opp}.$
\item A right (or left) \emph{comodule} is a right (resp. left) module in $\C^{\op}.$
\item A \emph{braided $V$-module morphism} is a morphism $\varphi$ between braided $V$-modules $(M,\rho)$ and $(N,\pi)$ such that $\varphi \circ \rho =\pi \circ (\varphi \otimes \Id_V):M\otimes V \rightarrow N.$
\end{itemize}
\end{definition}

Start as usual with a trivial example: in a preadditive category, any object $M$ equipped with the zero map $M\otimes V \rightarrow M$ is a module over any pre-braided object $(V,\sigma).$ We further interpret our new notion in more complicated settings from section \ref{sec:examples}.

\begin{example}\label{ex:BrMod}
\begin{enumerate}
\item One recovers the notion of (anti)commuting operators on $M$ when $\sigma = \pm \tau.$
\item Take $\C = \Set,$ and as a pre-braiding on a set $S$ take $\sigma_{\lhd}$ from \eqref{eqn:RackBraid}, coming from a self-distributive operation $\lhd.$ Condition \eqref{eqn:BrMod} becomes
$$(m \lhd a) \lhd b = (m \lhd b) \lhd  (a \lhd b) \qquad \forall m \in M, a,b \in S, $$
which defines precisely a \textbf{\textit{rack module}}  (= the \textit{rack-set} from \cite{Kamada3}, or the \textit{shadow} from \cite{ChangNelson}), having a knot-theoretical motivation. 
\item Any UAA $(V,\mu,\nu)$ in $\C$ comes with the pre-braiding $\sigma_{Ass}$ from \eqref{eqn:sigmaUAA}. Take a right module $(M,\rho)$ which we suppose \textbf{\textit{normalized}} here, i.e. 
\begin{equation}\label{eqn:NormalAction}
\rho \circ (\Id_M \otimes \nu) = \Id_M
\end{equation}
(morally, ``the unit acts by identity''). Condition \eqref{eqn:BrMod} becomes
$$\rho \circ (\rho \otimes \Id_V)=\rho \circ (\Id_M\otimes \mu),$$
giving the familiar notion of \textbf{\textit{module over an associative algebra}}.
\item Take a ULA $(V,[,],\nu)$ in a symmetric preadditive category $\C$. Endow $V$ with the braiding $\sigma_{Lei}$ from \eqref{eqn:sigmaULA}. Take a normalized right module $(M,\rho).$ Condition \eqref{eqn:BrMod} becomes
$$\rho \circ (\rho \otimes \Id_V)=\rho \circ (\rho \otimes \Id_V)\circ (\Id_M\otimes c_{V,V})+ \rho \circ (\Id_M\otimes [,]),$$
giving the familiar notion of \textbf{\textit{module over a Leibniz algebra}} (cf. \cite{Cyclic}).
\end{enumerate} 
\end{example} 

Note that, dually, left modules over associative or left Leibniz algebras are particular cases of left modules over pre-braided objects.

\medskip
Now, returning to the general monoidal category setting, try a special choice of $M$:
\begin{lemma}
Take a pre-braided object $(V, \sigma)$ in $\C.$ For a morphism $\epsilon : V=\II \otimes V = V \otimes \II \rightarrow \II,$ the following conditions are equivalent:
\begin{enumerate}
\item $\epsilon$ defines a right braided $V$-module $\II$;
\item $\epsilon$ defines a left braided $V$-module $\II$;
\item $\epsilon$ is a braided character.
\end{enumerate}
\end{lemma}

 This observation can be generalized to endow each tensor power of $V$ with a braided $V$-module structure using a braided character $\epsilon.$ Recall notations $\varphi_i$ from \eqref{eqn:phi_i}.

\begin{proposition}\label{thm:adjoint_module}
Given a pre-braided object $(V, \sigma)$ with a braided character $\epsilon,$ the map
$${^\epsilon}\!\pi:=\epsilon_1 \circ \ssigma_{V^{\otimes n},V} : V^{\otimes n} \otimes V\rightarrow V^{\otimes n}$$
defines a right braided $V$-module structure on $V^{\otimes n}.$ The braiding $\sigma$ is extended here to arbitrary powers of $V$ as in remark \ref{rmk:br_tensor}.
\end{proposition}
\begin{proof}
The definition of ${^\epsilon}\!\pi$ and repeated application of the YBE give
\begin{align*}
{^\epsilon}\!\pi \circ ({^\epsilon}\!\pi \otimes \Id_V) \circ (\Id_n \otimes \sigma) &= \\
(\epsilon \otimes \epsilon \otimes \Id_n)\circ \ssigma_{V^n,V^2} \circ (\Id_n \otimes \sigma) &= \\
((\epsilon \otimes \epsilon)\circ \sigma) \otimes \Id_n) \circ \ssigma_{V^n,V^2} & 
\end{align*}
which, by the definition of braided character, is the same as
$$(\epsilon \otimes \epsilon \otimes \Id_n) \circ \ssigma_{V^n,V^2} =
{^\epsilon}\!\pi \circ ({^\epsilon}\!\pi \otimes \Id_V). $$
The reader is advised to draw some diagrams to better follow the proof.
\end{proof}

\begin{definition}
We call the modules $(V^{\otimes n}, {^\epsilon}\!\pi) $ \emph{adjoint}.
\end{definition}

One recognizes the arrow operations ${^\epsilon}\!\pi_w$ from subsection \ref{sec:HomOper}. See that subsection for diagrams and some properties. In particular, along the lines of proposition \ref{thm:hom_operations}, one proves

\begin{proposition}\label{crl:AdjActOnComplexes}
The action ${^\epsilon}\!\pi$ on $T(V)$ intertwines the left braided differential ${^\xi}\!d$ for $\sigma$-compatible braided characters $\epsilon$ and $\xi.$ In other words, ${^\xi}\!d$ is a braided $V$-module morphism, for the adjoint braided $V$-module structure on $V^{\otimes n}.$
\end{proposition}

The motivation for our term comes from examples, where one recognizes familiar actions on $T(V)$ (cf. table \ref{tab:BasicIngredients}).

\medskip
We have seen that a module over a pre-braided object is a generalization of a braided character. Observe that this generalization picks the right property for a generalized version of theorem \ref*{thm:BraidedSimplHomCat} (where we replace the braided $V$-module $\II$ by arbitrary braided modules) to hold:

\medskip
\textbf{Theorem \ref*{thm:BraidedSimplHomCat}$^\mathbf{coeffs}$.}
\textit{
Let $(\C,\otimes,\II)$ be a preadditive monoidal category, $(V,\sigma)$ a pre-braided object in $\C,$ and $(M,\rho)$ and $(N,\lambda)$ a right and a left braided $V$-modules respectively. Then 
\begin{align*}
({^\rho}\! d)_n &:= (\rho \otimes \Id_{n-1}\otimes \Id_N)\circ (\Id_M\otimes\Csh^{1,n-1} \otimes \Id_N),\\ (d^\lambda)_n &:=(-1)^{n-1} (\Id_M \otimes \Id_{n-1}\otimes \lambda)\circ (\Id_M\otimes\Csh^{n-1,1} \otimes \Id_N),
\end{align*}
define a bidegree $-1$ tensor bidifferential for $V$ with coefficient in $M$ and $N.$
}

\medskip
The complicated expression \emph{a bidegree $-1$ tensor bidifferential for $V$ with coefficient in $M$ and $N$} hides what one naturally expects: two families of morphisms $d_n, d'_n: M \otimes V^{n}\otimes N \rightarrow M \otimes V^{n-1}\otimes N$ satisfying \eqref{eqn:bidiff}.

Pictorially, $({^\rho}\! d)_n$  is for example a signed sum of terms of the form
\begin{center}
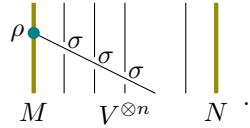

\begin{tikzpicture}[scale=0.4]
 \draw[olive,ultra thick] (0,0) -- (0,3);
 \draw (1,0) -- (1,1.3);
 \draw (1,1.7) -- (1,3);
 \draw (2,0) -- (2,0.8);
 \draw (2,1.2) -- (2,3);
 \draw (3,0) -- (3,0.3);
 \draw (3,0.7) -- (3,3);
 \draw (4,0) -- (0,2);
 \draw (5,0) -- (5,3);
 \draw[olive,ultra thick] (6,0) -- (6,3);
 \node at (2.8,0.7) [right]{$\sigma$};
 \node at (0.8,1.7) [right]{$\sigma$};
 \node at (1.8,1.2) [right]{$\sigma$};
 \node at (0,2) [left]{$\rho$};
 \fill[teal] (0,2) circle (0.2);
 \node at (0,0) [below] {$M$};
 \node at (6,0) [below] {$N$};
 \node at (3,0) [below] {$V^{\otimes n}$};
 \node at (7,0) [below] {.};
\end{tikzpicture}
   \captionof{figure}{Braided differentials with coefficients}\label{pic:BrDiffCoeffs}
\end{center}

The proof of this result is a direct generalization of that of theorem \ref{thm:BraidedSimplHomCat}. Moreover, all the remaining points of that theorem can be generalized to ``coefficient'' settings.

\begin{remark}\label{rmk:one-sided_coeffs}
Taking as $M$ or $N$ the unit object $\II$ with the zero module structure, one obtains a degree $-1$ tensor differential for $V$ with coefficient in the left braided $V$-module $N$ (resp. right braided $V$-module $M$) only.
\end{remark}

As usual, everything described here can be dualized, in any of the three senses described in subsections \ref{sec:co} and \ref{sec:r-l}.

\medskip

Adjoint modules also admit a version with coefficients:

\begin{proposition}\label{thm:adjoint_module_coeff}
Given a pre-braided object $(V, \sigma)$ and a right braided $V$-module $(M,\rho),$ the morphisms
$${^\rho}\!\pi:=\rho_1\circ (\Id_M \otimes \ssigma_{V^{n},V}): 
M \otimes V^{\otimes n}\otimes V\rightarrow M \otimes V^{\otimes n}$$
define a right braided $V$-module structure on $M \otimes V^{\otimes n},$ intertwining the left differential ${^\rho}\!d.$
\end{proposition}

In other words, ${^\rho}\!d$ is a braided $V$-module morphism.

\medskip
Having the Hochschild homology in mind, one also needs the notion of bimodules.

\begin{definition}
A \emph{bimodule} over a pre-braided object $(V, \sigma)$ is an object $M \in \Ob (\C)$ equipped with two morphisms $\rho:M\otimes V \rightarrow M$ and $\lambda:V\otimes M \rightarrow M,$ turning $M$ into a right and left modules respectively and satisfying the following compatibility condition:
$$\rho \circ (\lambda \otimes \Id_V)=\lambda \circ (\Id_V \otimes \rho):V\otimes M\otimes V\rightarrow M.$$
\end{definition}

Another interpretation of bimodules -- in terms of modules over appropriate pre-braided systems -- is given in \cite{Lebed}.

The bidifferential structure from theorem \ref*{thm:BraidedSimplHomCat}$^\mathbf{coeffs}$ can be nicely adapted to bimodules:

\begin{proposition}\label{thm:bimod}
 Let $(\C,\otimes,\II,c)$ be a symmetric preadditive category, $(V,\sigma)$ a pre-braided object in $\C,$ and $(M,\rho,\lambda)$ a bimodule over $V$. Then the families of morphisms 
\begin{align*}
({^\rho}\! d)_n &:= (\rho \otimes \Id_{n-1})\circ (Id_M\otimes\Csh^{1,n-1}),\\
(d^\lambda)_n &:=(-1)^{n-1}c_{M,V^{n-1}}^{-1}\circ (\Id_{n-1}\otimes \lambda)\circ (\Csh^{n-1,1} \otimes \Id_M)\circ c_{M,V^{n}},
\end{align*}
define a bidegree $-1$ tensor bidifferential for $V$ with coefficients in $M$ on the left.
\end{proposition}

By definition, $(d^\lambda)_n$ is a signed sum of terms of the form
\begin{center}
\begin{tikzpicture}[scale=0.3]
 \draw[olive,ultra thick,rounded corners] (0,0) -- (0,0.5) -- (5,1) -- (5,4)--(0,4.5)--(0,5);
 \draw (1,0) -- (1,5);
 \draw (3,0) -- (3,5);
 \draw (4,0) -- (4,5);
 \draw [rounded corners] (2,0) -- (2,2) --(2.8,2.4);
 \draw (3.2,2.6) -- (3.8,2.9);
 \draw (4.2,3.1) -- (5,3.5);
 \node at (5,3.5) [right]{$\lambda$};
 \fill[teal] (5,3.5) circle (0.2);
 \node at (0,0) [below] {$M$};
 \node at (1,0) [below] {$V$};
 \node at (2,0) [below] {$V$};
 \node at (3,0) [below] {$V$};
 \node at (4,0) [below] {$V$};
 \node at (6,0) {.};
\end{tikzpicture}
   \captionof{figure}{Braided differentials with bimodule coefficients}\label{pic:BrDiffBimod}
\end{center}

\begin{proof}
Relations $({^\rho}\! d)_{n-1}\circ ({^\rho}\! d)_n =0$ and
\begin{align*}
(d^\lambda)_{n-1}\circ (d^\lambda)_n &= c_{M,V^{n-2}}^{-1} \circ (d'^\lambda)_{n-1}\circ c_{M,V^{n-1}}\circ c_{M,V^{n-1}}^{-1}\circ (d'^\lambda)_n  \circ c_{M,V^{n}}\\
& = c_{M,V^{n-2}}^{-1} \circ (d'^\lambda)_{n-1}\circ  (d'^\lambda)_n  \circ c_{M,V^{n}} =0,
\end{align*}
with $(d'^\lambda)_n:=(-1)^{n-1}(\Id_{n-1}\otimes \lambda)\circ (\Csh^{n-1,1} \otimes \Id_M),$ follow directly from the corresponding identities in theorem \ref*{thm:BraidedSimplHomCat}$^\mathbf{coeffs}$. To prove the compatibility between $({^\rho}\! d)_n$ and $(d^\lambda)_n,$ observe that
$$(d^\lambda)_n =(-1)^{n-1}((\lambda\circ c_{M,V}) \otimes\Id_{n-1}) \circ (\Id_M\otimes c_{V,V^{n-1}}^{-1})\circ (\Id_M\otimes\Csh^{n-1,1}),$$
then use the defining property of bimodule, the naturality of $c$ and the YBE for $\sigma.$
\end{proof}

\begin{remark}
We have kept the notation $c^{-1},$ redundant for a symmetric $c,$ to be able to treat the non symmetric situation. In this case, on the picture showing $(d^\lambda)_n,$ the thick line (corresponding to $M$) should go behind all normal lines, in order to distinguish $c$ from $c^{-1}.$ One should be careful to differentiate the two pre-braidings, $c$ and $\sigma,$ which is difficult to do pictorially. For the above theorem to be still valid, one should change the compatibility condition defining a bimodule to the following one, different from the old one in general:
$$\lambda \circ (\Id_V \otimes \rho)\circ c_{M\otimes V,V} =\rho \circ (\lambda \otimes \Id_V)\circ c_{M,V}\circ c_{V,V}^{-1}:M\otimes V\otimes V\rightarrow M,$$
\begin{center}
\begin{tikzpicture}[scale=0.4]
 \draw[olive,ultra thick] (0,0) -- (0,1.8);
 \draw[olive,ultra thick] (0,2.2) -- (0,4.5);
 \draw (1,0) -- (0,1);
 \draw [rounded corners] (2,0) -- (-0.5,2.5) --(0,3);
 \node at (0,3) [left]{$\lambda$};
 \fill[teal] (0,3) circle (0.2);
 \node at (0,1) [left]{$\rho$};
 \fill[teal] (0,1) circle (0.2);
 \node at (0,0) [below] {$M$};
 \node at (1,0) [below] {$V$};
 \node at (2,0) [below] {$V$};
 \node at (4,2) {$=$};
\end{tikzpicture}
\begin{tikzpicture}[scale=0.4]
 \node at (-2,2) {};
 \draw[olive,ultra thick] (0,0) -- (0,1.8);
 \draw[olive,ultra thick] (0,2.2) -- (0,4.5);
 \draw [rounded corners](1,0) -- (1,3)-- (0,4);
 \draw [rounded corners] (2,0) -- (1.2,0.8);
 \draw [rounded corners] (0.8,1.2) -- (-0.5,2.5) --(0,3);
 \node at (0,3) [left]{$\lambda$};
 \fill[teal] (0,3) circle (0.2);
 \node at (0,4) [left]{$\rho$};
 \fill[teal] (0,4) circle (0.2);
 \node at (0,0) [below] {$M$};
 \node at (1,0) [below] {$V$};
 \node at (2,0) [below] {$V$};
\end{tikzpicture}
  \captionof{figure}{Bimodules in the non symmetric case}\label{pic:BimodNonSymm}
\end{center}
All the crossings correspond to the braiding $c$ here.
\end{remark}

A more elegant solution for the non symmetric case would be welcome.

\medskip
We finish this section with examples, interpreting several classical homology theories with coefficients as braided ones:

\begin{example}\label{ex:BrModHom}
\begin{enumerate}
\item 
Taking a vector space $V$ with a simple flip as a braiding and, for instance, its symmetric algebra $S(V)$ as a module over $V$ (with the action coming from concatenation, as usual), one obtains more complicated versions of the Koszul complex.
\item
In the case of shelves, one recovers the \textbf{\textit{shelf and rack homologies with coefficients}}, hinted at in \cite{ChangNelson}. 
\item
For Leibniz algebras, our machinery gives the \textbf{\textit{Leibniz homology with coefficients}}, generalizing the Chevalley-Eilenberg homology (cf. \cite{Cyclic}). 

In these three cases one generally puts the coefficients only on the left (cf. remark \ref{rmk:one-sided_coeffs}).
 
\item
Coefficients on both sides turn out to be particularly useful for associative algebras in a symmetric preadditive category. In this setting, proposition \ref{thm:bimod} gives the following differential for an algebra bimodule $(M,\rho,\lambda)$:
\begin{align*}
({^\rho}\! d -d^\lambda)_n & := \rho_1+\sum_{i=1}^{n-1}(-1)^i \mu_{i+1} +(-1)^n \lambda_{1} \circ c_{M\otimes V^{n-1},V}\\
&+ \text{some terms involving  } \nu.
\end{align*}
For $\C = \Vect,$ one can get rid of the terms with $\nu$ as it was done in the proof of point \ref*{it:GroupHom} of proposition \ref{thm:ComputAss}, getting the \textbf{\textit{Hochschild differential}}. 

\item The co-version of the previous differential is the \textbf{\textit{Cartier differential}} for coalgebras (cf. \cite{Cartier}, where it was first introduced).
\end{enumerate}
\end{example}

\section*{Acknowledgements}
I would like to thank Marc Rosso for sharing his passion for quantum shuffles and for his patient encouragments, and J{\'o}zef Przytycki for introducing me to his work on distributive homology. I am also grateful to Paul-Andr\'e Melli\`es for his interest in my work, and to Arnaud Mortier for comments on a preliminary version of this paper. My deep gratitude goes to Jean-Louis Loday whom I have never had the chance to meet in person, but whose mathematics I have always admired.

\bibliographystyle{plain}
\bibliography{biblio}

\end{document}